\documentclass[a4paper,12pt]{article}

\usepackage{graphics}
\usepackage{epsf}

\usepackage{amsfonts}
\usepackage{amssymb}
\usepackage{epsfig}
\usepackage[latin1]{inputenc}
\usepackage{color}

\usepackage{amsthm,amsmath,amssymb}

\newtheorem{theorem}{Theorem}

\newtheorem{remark}[theorem]{Remark}

\begin{document}
%
\title{Order reduction and how to avoid it when Lawson methods integrate reaction-diffusion boundary value problems}
\author{{\sc B. Cano} \thanks{Corresponding author. Email: bego@mac.uva.es}  \\ \small
IMUVA, Departamento de Matem\'atica Aplicada,\\ \small Facultad de
Ciencias, Universidad de
Valladolid,\\ \small Paseo de Bel\'en 7, 47011 Valladolid,\\ \small Spain \\
{\sc and}\\
{\sc N. Reguera}\thanks{Email: nreguera@ubu.es} \\
\small IMUVA, Departamento de Matem\'aticas y Computaci\'on, \\
\small  Escuela Polit\'ecnica Superior, Universidad de Burgos,\\
\small  Avda. Cantabria, 09006 Burgos, \\ \small  Spain }
\date{}

%
%

\maketitle

\begin{abstract}
It is well known that Lawson methods suffer from a severe order reduction when integrating initial boundary value problems where the solutions are not periodic in space or do not satisfy enough conditions of annihilation on the boundary. However, in a previous paper,
a modification of Lawson quadrature rules has been suggested so that no order reduction turns up when integrating linear problems subject to even time-dependent boundary conditions. In this paper, we describe and thoroughly analyse a technique to avoid also order reduction when integrating nonlinear problems. This is very useful because, given any Runge-Kutta method of any classical order, a Lawson method can be constructed associated to it for which the order is conserved.
\end{abstract}




\section{Introduction}
There is an effort in the recent literature to understand the order reduction which turns up when integrating initial boundary value problems subject to non-periodic boundary conditions with exponential methods. In that sense, we mention the papers \cite{acr1,CM,FOS} which correspond to the integration of linear differential problems through Lawson and standard exponential quadrature rules and also through splitting integrators. Understanding this well has led to the design of techniques to avoid that order reduction for both linear and nonlinear problems when the boundary conditions are not only non-periodic but also time-dependent. In that way, for linear problems, suitable modifications of Lawson and exponential quadrature rules have been seen to lead to the order of the underlying classical quadrature rule, as high as desired \cite{acr2,CM}. However, with exponential splitting methods, just order $2$ has been achieved for the moment for both linear and nonlinear problems \cite{ACR,acrnl,CR,EO,EO2}.

As Lawson methods of classical order as high as desired can be constructed directly from any chosen Runge-Kutta method of that order \cite{L}, our aim in this paper is to analyse the order reduction which turns up when integrating nonlinear reaction-diffusion boundary value problems with Lawson methods and to suggest a technique to avoid it also in such a case. For that, we consider an analysis of both the space and time discretization and we firstly give results on the type of error which is obtained when integrating firstly in space and then in time with both vanishing and non-vanishing boundary conditions. As for linear problems with vanishing boundary conditions \cite{acr1}, local order $1$ is just observed in general, which leads to global order also $1$ when a summation-by-parts argument can be applied. With non-vanishing boundary conditions, two types of behaviour are observed, which are both useless from the practical point of view because, either the error diminishes with the timestepsize but grows with the space grid, or remains practically the same when both the space and time grid diminish.

On the other hand, with our proposal, when integrating firstly in time, some terms of Lawson method are calculated through the solution of linear initial boundary value problems for which suitable boundary values are suggested.
Then, the space discretization of those problems must be performed. The main achievement of the paper is how to suggest those boundary values in order to get a given accuracy and how to calculate them in terms of the given data of the problem. More precisely, with any consistent RK method, we prove that local order $2$ can be achieved by calculating the boundary values exactly in terms of the given data for Dirichlet boundary conditions or in an approximate way (through the numerical solution) for Neumann/Robin ones. In any case, to obtain local order $2$, there is no need to resort to numerical differentiation. Besides, if the RK method has order $\ge 2$, global order $2$  is achieved if the summation-by-parts argument is applied. (For that, apart from enough regularity, we need the assumptions (\ref{spp}),(\ref{cspc}),(\ref{cspch}), which seem to be true but which proof is out of the scope of this paper.) Moreover, by resorting to numerical differentiation, local order $3$ can be achieved and, whenever a CFL condition is satisfied (\ref{condcfl}), global error $\ge 2$ and smaller errors are obtained. (We remark that this condition is not very restrictive in the sense that, when the differential problem is of second-order in space, $\gamma=1$ in (\ref{condcfl}) and therefore it just means that the timestepsize is not too big with respect to the space grid.) Finally, although the formulas get more complicated, by using numerical differentiation, local order $4$ can also be achieved if the underlying RK method is of order $\ge 3$. We have not described the formulas to get local order $\ge 5$ because they get more and more complicated, because there are less problems where so much accuracy is required, and also due to the fact that numerical differentiation is badly-posed \cite{JSS}, and therefore the rounding errors associated to their use might cause that not so much accuracy is achieved.

The paper is structured as follows. Section 2 gives some preliminaries on the abstract framework for the problem, on Lawson methods and on the assumptions we make for the space discretization. Section 3 analyses thoroughly the local and global error with the classical approach. The modification to avoid order reduction is suggested in Section 4, separating the cases in which local orders 2,3 and 4 want to be achieved. The final formulas to be implemented are respectively (\ref{fd2}), (\ref{fd3}) and (\ref{fd4}) and, in Remarks \ref{rembc}, \ref{rembc3} and \ref{rembc4}, a thorough discussion is given on how to calculate the required suggested boundaries for either Dirichlet, Robin or Neumann boundary conditions. Finally, some numerical results are shown in Section 5 which corroborate the theoretical results.

\section{Preliminaries}
Let $X$ and $Y$ be Banach spaces and let $A:D(A)\subset X \to X$ and
$\partial: X \to Y$ be linear operators. Our goal is to avoid order reduction when integrating in time through Lawson methods the nonlinear abstract non homogeneous initial
boundary value problem
\begin{eqnarray}
\label{laibvp}
\begin{array}{rcl}
u'(t)&=&Au(t)+f(t,u(t)), \quad  0\le t \le T,\\
u(0)&=&u_0 \in X,\\
\partial u(t)&=&g(t)\in Y, \quad  0\le t \le T.
\end{array}
\end{eqnarray}
We will assume that the functions $f:[0,T]\times X \to X$ (in general nonlinear)
and $g: [0,T] \to Y$ are  regular enough.

The abstract setting (\ref{laibvp}) permits to cover a wide range
of nonlinear evolutionary problems governed by partial
differential equations. We use the following hypotheses, similar to the ones in \cite{acrnl} when avoiding the same kind of problems with exponential splitting methods.

\begin{enumerate}
\item[(A1)] The boundary operator $\partial:D(A)\subset X\to Y$ is
onto and $g\in C^1([0,T],Y)$.

\item[(A2)] Ker($\partial$) is dense in $X$ and
$A_0:D(A_0)=\ker(\partial)\subset X \to X$, the restriction of $A$
to Ker($\partial$), is the infinitesimal generator of a $C_0$-
semigroup $\{e^{t A_0}\}_{t\ge 0}$ in $X$, which type $\omega$ is
assumed to be negative.

\item[(A3)] If $z \in \mathbb{C}$ satisfies $\Re (z) >0$ and $v
\in Y$, then the steady state problem
\begin{eqnarray}
 Ax &=& zx,\\
 \partial x&=&v,
\label{stationaryproblem}
\end{eqnarray}
possesses a unique solution denoted by $x=K(z)v$. Moreover, the
linear operator $K(z): Y \to D(A)$ satisfies
\begin{eqnarray}
\label{stationaryoperator} \| K(z)v\| \le C\|v\|,
\end{eqnarray}
where the constant $C$ holds for any $z$ such that $Re (z) \ge
\omega_0 > \omega$.

\item[(A4)] The nonlinear source $f$ belongs to $C^1([0,T] \times
X, X)$.

\item[(A5)] The solution $u$ of (\ref{laibvp}) satisfies $u\in
C^1([0,T], X)$, $u(t) \in D(A)$ for all $t \in [0,T]$ and $Au \in C([0,T], X)$.

\end{enumerate}

In the remaining of the paper, we always suppose that (A1)-(A5)
are satisfied. However, we notice that we also assume more
regularity in certain results. As for the well-posedness of problem (\ref{laibvp}), the same remark as in \cite{acrnl} can be made. We repeat it here for the sake of clarity.

\begin{remark}
From hypotheses (A1)-(A4), the problem (\ref{laibvp}) with
homogeneous boundary conditions has a unique classical solution
for small enough time intervals (see Theorem 6.1.5 in
\cite{pazy}).

Regarding the nonhomogeneous case, as in (A1) we also assume that  $g \in C^1([0,T],Y)$,  we can
look for a solution of (\ref{laibvp}) given by:
\begin{eqnarray*}
u(t)=v(t)+K(z)g(t), \quad t \ge 0,
\end{eqnarray*}
for some fixed $\Re(z) >\omega$. Then, $v$ is solution of an IBVP
with vanishing boundary values similar to the one in \cite{pazy}
and the well-posedness  for the case of
nonhomogeneous boundary values is a direct consequence if we take
the abstract theory for initial boundary value problems in
\cite{alonsomallop, palenciaa} into account.

However, condition (A4) may be very strong. When  $X$ is a
function space with a norm  $L^p$, $ 1 \le p < +\infty$, and $f$
is a Neminskii operator,
\begin{eqnarray*}
u \to f(u)= \phi (u),
\end{eqnarray*}
with $\phi : \mathbb{C} \to \mathbb{C}$, (A4)
implies that $\phi$ is globally Lipschitz in $\mathbb{C}$. This
objection disappears by considering the supremum norm, which is
used in our numerical examples, where the nonlinear source is
given by
\begin{eqnarray}
\label{nonlinearterm} u \to f(t,u)= \phi (u) + h(t),
\end{eqnarray}
with $h:[0,T] \to X$, that is, $f$ is the sum of a Neminskii
operator and a linear term. In this way, problem (\ref{laibvp}) is
well posed.
\end{remark}

Because of hypothesis (A2), $\{\varphi_j(t A_0)\}_{j=1}^{3}$   are
bounded operators for $t>0$, where $\{\varphi_j\}$ are the
standard functions which are used in exponential methods \cite{HO2} and which
are defined by
\begin{eqnarray}
\varphi_j( t A_0)=\frac{1}{t^j} \int_0^t
e^{(t-\tau)A_0}\frac{\tau^{j-1}}{(j-1)!}d\tau, \quad j \ge 1.
\label{varphi}
\end{eqnarray}
It is well-known that they can be calculated in a recursive way through the formulas
\begin{eqnarray}
\varphi_{j+1}(z)=\frac{\varphi_j(z)-1/j!}{z}, \quad z \neq 0,
\qquad \varphi_{j+1}(0)=\frac{1}{(j+1)!}, \qquad \varphi_0(z)=e^z.
\label{recurf}
\end{eqnarray}

For the time integration, we will center on exponential
Lawson methods which are determined by an explicit Runge-Kutta tableau and which, when applied to the
finite-dimensional nonlinear problem  like
\begin{eqnarray}
U'(t) = M U(t)+F(t,U(t)), \label{linfd}
\end{eqnarray}
where $M$ is a matrix, read like this at each step
\begin{eqnarray}
K_{n,i}&=&e^{c_i k M}U_n+k \sum_{j=1}^{i-1} a_{ij} e^{(c_i-c_j)k M} F(t_n+c_j k, K_{n,j}), \quad i=1,\dots,s, \label{etapas} \\
U_{n+1}&=&e^{k M}U_n+k \sum_{i=1}^s b_i e^{(1-c_i)k M}F(t_n+c_i k,K_{n,i}),
\label{lawson}
\end{eqnarray}
where $k >0$ is the time stepsize and $t_n=t_0+nk$.

Following the example in Section 2 of \cite{acrnl}, we take
$X=C(\overline{\Omega})$ for a certain bounded domain $\Omega\in \mathbb{R}^d$. There, we consider the maximum norm and  a
certain grid $\Omega_h$ (of $\Omega$) over which the approximated
numerical solution will be defined. In this way, this numerical
approximation belongs to $C^N$, where $N$ is the number of nodes
in the grid, and the maximum norm $\|u_h\|_h=\|[u_1,
\ldots, u_N]^T\|_h= \max_{1 \le i \le N}|u_i|$ is considered.

Notice that, usually, when considering Dirichlet boundary
conditions, nodes on the boundary are not considered while, when
using Neumann or Robin boundary conditions, the nodes on the
boundary are taken into account.

In that sense, we consider the projection operator
\begin{eqnarray}
 P_h : X \to \mathbb{C}^N, \label{ph}
\end{eqnarray}
which takes a function to its values over the grid $\Omega_h$.
 On the other hand, the
operator $A$, when applied over functions which satisfy a certain
condition on the boundary $\partial u=g$, is discretized by
\begin{eqnarray*}
A_{h,0}U_h+C_h g,
\end{eqnarray*}
where $A_{h,0}$ is the matrix which discretizes $A_0$ and $C_h: Y
\to \mathbb{C}^N$ is another operator, which is the one which
contains the information on the boundary.

We also assume that the source function $f$ has also sense as
function from $[0,T]\times \mathbb{C}^N$ on $\mathbb{C}^N$ and,
for each $t \in [0,T]$ and $u \in X$,
\begin{eqnarray}
\label{fandph} P_hf(t,u)= f(t, P_hu).
\end{eqnarray}
This fact is obvious when $f$ is given by (\ref{nonlinearterm}).

In a similar way to \cite{acrnl}, we consider the following hypotheses:
\begin{enumerate}
\item[(H1)] The matrix $A_{h,0}$ satisfies
\begin{enumerate}
\item $\|e^{tA_{h,0}}\|_h \le 1$,
\item $A_{h,0}$ is invertible and $\|A_{h,0}^{-1}\|_h \le C$ for
some constant $C$ which does not depend on $h$.
\end{enumerate}
\item[(H2)] We define the elliptic projection $R_{h}:D(A) \to
\mathbb{C}^N$  as the solution of
\begin{eqnarray}
A_{h,0} R_{h} u+ C_{h}\partial u=P_h Au. \label{rh}
\end{eqnarray}
We assume that there exists a subspace $Z \subset D(A)$, such
that, for $u \in Z$,
\begin{enumerate}
\item[(a)] $A_0^{-1} u \in Z$ and $e^{tA_0}u, f(t,u) \in Z$, for $t \in [0,T]$.
\item[(b)] for some $\varepsilon_{h}$ and $\eta_{h}$ which are
both small with $h$,
\begin{equation}
\label{consistency} \hspace{-0.5cm} \left\| A_{h,0}({P_hu-R_{h}u})
\right\| \le \varepsilon_{h} \left\| u \right\|_Z, \quad
\left\|P_hu-R_{h}u \right\| \le \eta_{h} \left\| u \right\|_Z.
\end{equation}
(Although obviously, because of (H1b), $\eta_h$ could be taken as
$C\varepsilon_h$, for some discretizations $\eta_h$ can decrease more quickly with $h$
than $\varepsilon_h$ and that leads to better error bounds in the
following sections.)
\item[(c)] $\|A_{h,0}^{-1} C_h\|_h \le C''$ for some constant
$C''$ which does not depend on $h$. This resembles the continuous
maximum principle which is satisfied because of
(\ref{stationaryoperator}) when $z=0$.
\end{enumerate}
\item[(H3)] The nonlinear source $f$ belongs to
$C^1([0,T] \times \mathbb{C}^N, \mathbb{C}^N)$ and the derivative with respect to the variable in $\mathbb{C}^N$
is uniformly bounded in a neighbourhood of the solution where the numerical approximation stays.
\end{enumerate}

As in \cite{acrnl}, hypothesis (H1a) can be deduced in our numerical experiments
by using the logarithmic norm of matrix
$A_{h,0}$.

\section{Classical approach: Discretizing firstly in space and then in time}

It was already proved in \cite{acr1} that Lawson methods, even when applied to linear problems with vanishing boundary conditions, show in general just order $1$ in time.
For that reason, with nonlinear problems and probably non-vanishing boundary conditions, in general we cannot expect more than that order. In any case, in this section we generalise the main result in \cite{acr1} not only in the sense of considering more general problems but also in the sense of taking the error coming from the space discretization into account.

Notice that, by using the space discretization which is described in the previous section, the following semidiscrete problem arises after
discretising (\ref{laibvp}),
\begin{eqnarray}
\label{laibvpdisc} \left.
\begin{array}{rcl}
U'_h(t)&=&A_{h,0}U_h(t) + C_{h}g(t) +f(t,U_h(t)),\\
U_h(t_0)&=&P_h u(t_0).
\end{array}
\right.
\end{eqnarray}
Then, applying Lawson method (\ref{etapas})-(\ref{lawson}) to this, the following formulas define one step from $U_h^n$ to $U_h^{n+1}$:
\begin{eqnarray}
K_{h,i}^n &=& e^{c_i k A_{h,0}}U_h^n+k \sum_{j=1}^{i-1} a_{ij} e^{(c_i-c_j)k A_{h,0}}[ C_h g(t_n+c_j k)+f(t_n+c_j k,K_{h,j}^n)], \quad i=1,\dots,s, \nonumber \\
U_h^{n+1}&=& e^{ k A_{h,0}}U_h^n +k \sum_{i=1}^s b_i e^{(1-c_i)k A_{h,0}}[C_h g(t_n+c_i k)+f(t_n+c_i k,K_{h,i}^n)]. \label{lawsonca}
\end{eqnarray}
We define the local error as $\rho_{h,n+1}=\bar{U}_h^{n+1} -P_h u(t_{n+1})$, where $u(t)$ is the solution of (\ref{laibvp}) and $\bar{U}_h^{n+1}$ is deduced as $U_h^{n+1}$ but starting from $P_h u(t_n)$ instead of $U_h^n$. Then,
\begin{eqnarray}
\bar{U}_h^{n+1}= e^{ k A_{h,0}}P_h u(t_n) +k \sum_{i=1}^s b_i e^{(1-c_i)k A_{h,0}}[C_h g(t_n+c_i k)+f(t_n+c_i k,\bar{K}_{h,i}^n)], \label{Uhb}
\end{eqnarray}
where, for $i=1,\dots,s$,
\begin{eqnarray}
\bar{K}_{h,i}^n = e^{c_i k A_{h,0}}P_h u(t_n) +k \sum_{j=1}^{i-1} a_{ij} e^{(c_i-c_j)k A_{h,0}}[ C_h g(t_n+c_j k)+f(t_n+c_j k,\bar{K}_{h,j}^n)]. \label{Khb}
\end{eqnarray}
\subsection{Local and global error when $g(t)\equiv 0$}
In this subsection, we will analyse the error with the classical approach under the assumption that the boundary conditions vanish, which is the only case which has been studied for other standard exponential  methods when integrating nonlinear problems and where just the error coming from the time integration has been analysed \cite{HO}. As distinct, here we also consider the error coming from the space discretization.
\begin{theorem}
\label{teorcalocal}
Under hypotheses (A1)-(A5) and (H1)-(H3),
whenever $u\in C([0,T],Z)$ and $\partial u\equiv 0$,
$$\rho_{h,n}=\bar{U}_h^{n+1}-P_h u(t_{n+1})=O(k), $$
where the constant in Landau notation is independent of $k$ and $h$. Moreover, if  $\sum_{i=1}^s b_i=1$ and $u\in C^2([0,T],X)$,
it happens that
$$A_{h,0}^{-1} \rho_{h,n}=O(\eta_h k+k^2).$$
\end{theorem}
\begin{proof}
Firstly we notice that, for $i=1,\dots,s$, $\bar{K}_{h,i}^n$ in (\ref{Khb})
are uniformly bounded on $h$ when $g(t)\equiv 0$, which can be proved by induction on $i$ by using (H1a),  (A4) and (A5).
Then, by using (\ref{recurf}),
\begin{eqnarray}
\bar{U}_h^{n+1}&=& P_h u(t_n) + k A_{h,0} \varphi_1(k A_{h,0}) P_h u(t_n)+k \sum_{i=1}^s b_i e^{(1-c_i)k A_{h,0}}f(t_n+c_i k,\bar{K}_{h,i}^n) \nonumber \\
&=& P_h u(t_n)+O(k), \nonumber
\end{eqnarray}
where we have used that $A_{h,0} \varphi_1(k A_{h,0})P_h u(t_n)$ is uniformly bounded on $h$ because
\begin{eqnarray}
A_{h,0} \varphi_1(k A_{h,0}) P_h u(t_n)& = &\varphi_1(k A_{h,0}) A_{h,0} R_h u(t_n)+\varphi_1(k A_{h,0}) A_{h,0}(P_h -R_h) u(t_n) \nonumber \\
& = &\varphi_1(k A_{h,0}) P_h A u(t_n)+O(\varepsilon_h). \label{acot0}
\end{eqnarray}
(Here, the second equality comes from the fact that $A_{h,0} R_h u(t)= P_h A u(t)$ due to (\ref{rh}) with $\partial u=0$, and also to (\ref{consistency}) considering that $u\in C([0,T],Z)$.)

As for the second result, notice that, by using (\ref{recurf}) again and an argument similar to (\ref{acot0}) ,  $\bar{K}_{h,i}^n$ can also be written as
\begin{eqnarray}
\bar{K}_{h,i}^n=P_h u(t_n)+c_i k A_{h,0} \varphi_1(c_i k A_{h,0})P_h u(t_n)+O(k)=P_h u(t_n)+O(k).
\label{kbhni}
\end{eqnarray}
Then, using now (\ref{recurf}) to expand $e^{k A_{h,0}}$ till $\varphi_2(k A_{h,0})$ and $e^{(1-c_i)k A_{h,0}}$ till $\varphi_1((1-c_i)k A_{h,0})$ and again an argument similar to (\ref{acot0}) for $A_{h,0}\varphi_2(k A_{h,0}) P_h u(t_n)$,
\begin{eqnarray}
\lefteqn{A_{h,0}^{-1} \rho_{h,n+1}} \nonumber \\
&=&A_{h,0}^{-1}\bigg[ P_h u(t_n)+ k A_{h,0} P_h u(t_n)+k^2 A_{h,0}^2 \varphi_2(k A_{h,0})P_h u(t_n) \nonumber \\
&& \hspace{1cm}+k \sum_{i=1}^s b_i [f(t_n+c_i k, \bar{K}_{h,i}^n)+(1-c_i)k A_{h,0} \varphi_1((1-c_i) k A_{h,0})f(t_n+c_i k, \bar{K}_{h,i}^n)] \nonumber \\
&&\hspace{1cm} -P_h u(t_n)-k P_h \dot{u}(t_n)+O(k^2)\bigg] \nonumber \\
&=& k [P_h u(t_n)-A_{h,0}^{-1} P_h \dot{u}(t_n)+A_{h,0}^{-1}f(t_n,P_h u(t_n))]+O(k^2), \label{formula}
\end{eqnarray}
where (H1b) and (\ref{kbhni}) have been used as well as the fact that $\sum_{i=1}^s b_i=1$.
Using now that
$$A_{h,0}^{-1}P_h \dot{u}(t)=A_{h,0}^{-1} P_h [A u(t)+f(t,u(t))]=R_h u(t)+A_{h,0}^{-1} P_h f(t,u(t)),$$
the bracket in (\ref{formula}) is $O(\eta_h)$ according to (\ref{consistency}) and using (\ref{fandph}), from what the result follows.
\end{proof}

 Using the classical argument for the global error, the first result of the previous theorem would not lead to convergence. However, by using the second result and a few more assumptions, a summation-by-parts argument very similar to that given in \cite{acrnl} for Strang method and parabolic problems when avoiding order reduction does lead to the following result.
\begin{theorem}
\label{teorcaglobal}
Under the hypotheses of Theorem \ref{teorcalocal}, and assuming also that $u\in C^3([0,T],X)\cap C^1([0,T],Z)$, with $\dot{u}(t)\in D(A)$ for $t\in [0,T]$, $\dot{u}\in C([0,T],Z)$, $A \dot{u}\in C([0,T],X)$  and
\begin{eqnarray}
\|k A_{h,0} \sum_{r=1}^{n-1} e^{r k A_{h,0}}\|_h
\le C , \quad 0 \le n k \le T, \label{spp}
\end{eqnarray}
it happens that
$$e_{h,n}= U_h^n-P_h u(t_n)=O(\eta_h +k).$$
\end{theorem}

\subsection{Local and global error when $g(t) \neq 0$}
In this subsection we study the more general case in which $g(t)\neq 0$. For the sake of brevity, we will consider just the case in which all the nodes $c_i$ are strictly increasing with the index $i$, but we will distinguish between the case in which $c_i\neq 1$ for $i=1,\dots,s$ and the case in which some of those $c_i$ is equal to $1$. Both types of results show a very poor behaviour, but our aim is just to explain the differences between them.
\begin{theorem}
\label{teorcalocal2}
Under hypotheses (A1)-(A5) and (H1)-(H3), assuming also that $\partial u\neq 0$, $c_1<c_2<\dots<c_s$, $c_i\neq 1$ for $i=1,\dots,s$ and that there exists $C'$, $h_0$ such that
\begin{eqnarray}
\|\tau A_{h,0} e^{\tau A_{h,0}}\|_h \le C', \quad \tau\ge 0,\quad h\le h_0,\label{acotses}
\end{eqnarray}
it happens that
$$\rho_{h,n}=O(1), \quad A_{h,0}^{-1} \rho_{h,n}=O(k),$$
where the constant in Landau notation is independent of $k$ and $h$.
\end{theorem}
\begin{proof}
Firstly we notice that now the stages $\bar{K}_{h,i}^n$ in (\ref{Khb})
are also uniformly bounded because, as the nodes are different,
$$ke^{(c_i-c_j)k A_{h,0}}C_h g(t_n+c_jk)=\frac{1}{(c_i-c_j)}[(c_i-c_j)k A_{h,0}e^{(c_i-c_j)k A_{h,0}}]A_{h,0}^{-1}C_h g(t_n+c_j k).$$
Then, the bracket is bounded because of (\ref{acotses}) and the last factor because of (H2c). With the same argument, as $c_i\neq 1$, $\bar{U}_h^{n+1}$ in (\ref{Uhb})
is also uniformly bounded, which implies that $\rho_{h,n}$ is just $O(1)$. On the other hand, by using (\ref{recurf}) in the first term of (\ref{Uhb}),
\begin{eqnarray}
A_{h,0}^{-1} \rho_{h,n+1}&=&A_{h,0}^{-1}\bigg[ P_h u(t_n)+k A_{h,0} \varphi_1(k A_{h,0}) P_h u(t_n)
\nonumber \\
&&+k \sum_{i=1}^s b_i e^{(1-c_i)k A_{h,0}}[f(t_n+c_i k,\bar{K}_{h,i}^n)+C_h g(t_n+c_i k)]-P_h u(t_n)+O(k)\bigg], \nonumber
\end{eqnarray}
which again is $O(k)$ because of (H2c).
\end{proof}
Now, with the same summation-by-parts argument as in \cite{acrnl}, the following result follows for the global error.
\begin{theorem}
\label{teorcaglobal2}
Under the same hypotheses of Theorem \ref{teorcalocal2}, if $u\in C^2([0,T],X)$ and (\ref{spp}) holds, then
$e_{h,n}=U_h^n-P_h u(t_n)=O(1)$.
\end{theorem}
Let us now consider the case in which some $c_i=1$.
\begin{theorem}
\label{teorcalocal3}
Under hypotheses (A1)-(A5) and (H1)-(H3), assuming also that $\partial u\neq 0$, $c_1<c_2<\dots<c_s$, $c_i=1$ for some $i\in \{1,\dots,s\}$ and (\ref{acotses}),
it happens that
$$\rho_{h,n}=O(1+k\|C_h\|), \quad A_{h,0}^{-1} \rho_{h,n}=O(k),$$
where the constant in Landau notation is independent of $k$ and $h$.
\end{theorem}
\begin{proof}
The proof is the same as that of Theorem \ref{teorcalocal2}, with the difference that now one of the terms in (\ref{Uhb}) can just be bounded by $k \|C_h\|$ because $c_i=1$ for some $i$.
However, the result for $ A_{h,0}^{-1} \rho_{h,n}$ is the same as in that theorem because of (H2c).
\end{proof}
As we will see in the numerical experiments, this explains that the local error behaves very badly, because it grows when $h$ diminishes. However, in spite of that, for fixed but small $h$, it behaves with order $1$ in $k$ because the term in $k\|C_h\|$ dominates.
The same happens with the global error, as the following theorem states by using again a summation-by-parts argument.
\begin{theorem}
\label{teorcaglobal3}
Under the same hypotheses of Theorem \ref{teorcalocal3}, if $u\in C^2([0,T],X)$ and (\ref{spp}) holds, then
$e_{h,n}=U_h^n-P_h u(t_n)=O(1+k\|C_h\|)$.
\end{theorem}

\section{Suggested approach: Discretizing firstly in time and then in space}

When discretizing firstly in time, some suggestion must be given for the exponential-type operators which turn up in Lawson formulas (\ref{etapas})-(\ref{lawson}). If $u(t_n)$  vanished at the boundary, it could seem natural to substitute those exponential operators by the $C_0$-semigroup $e^{\tau A_0}$ for suitable scalar values $\tau$. However, that may also lead to order reduction since the solution of
$$\dot{u}(\tau)=A_0 u(\tau), \quad u(0)=u_0,$$
cannot be accurately enough approximated by an expansion of the form
$$u_0+\tau A_0 u_0+\frac{\tau^2}{2}A_0^2 u_0+\cdots$$
unless $u_0\in D(A_0^l)$ for a high enough $l$. Besides, as $u(t_n)$ does not vanish in general,
suitable initial boundary value problems must be considered with the appropiate boundary values. We will consider three different cases depending on whether we want to achieve local order $2$, $3$ or $4$. For local order $2$ and Dirichlet boundary conditions, those boundaries can be calculated directly in terms of the data $f$ and $g$ of the problem (\ref{laibvp}). For the same order, but with Neumann/Robin boundary conditions, the approximation at the boundary given by the space discretization of the problem must be used. For higher orders, we will have to resort to numerical differentiation to calculate those boundaries.

From now on, we will assume that the coefficients of Butcher tableau satisfy the following standard equalities:
\begin{eqnarray}
\sum_{i=1}^s b_i=1, \qquad \sum_{j=1}^{i-1} a_{ij}=c_i, \,\, i=1,\dots,s. \label{cscond}
\end{eqnarray}

\subsection{Searching for local order $2$}
For the stages, starting from the continuous approximation $u_n$ at $t=t_n$, we consider recursively
\begin{eqnarray}
K_{n,i}&=&v_n(c_i k)+k\sum_{j=1}^{i-1} a_{ij} w_{n,j}((c_i-c_j)k), \quad i=1,\dots,s,\label{etapast}
\end{eqnarray}
where
\begin{eqnarray}
\left\{ \begin{array}{rcl} \dot{v}_n(s)&=& A v_n(s), \\ v_n(0)&=&u_n, \\ \partial v_n(s)&=&\partial u(t_n), \end{array}\right. \quad \left\{ \begin{array}{rcl} \dot{w}_{n,j}(s)&=& A w_{n,j}(s), \\ w_{n,j}(0)&=&f(t_n+c_jk,K_{n,j}), \\ \partial w_{n,j}(s)&=&0, \end{array}\right. \label{vwetapast}
\end{eqnarray}
Then, we suggest as the numerical approximation at the next step
\begin{eqnarray}
u_{n+1}=\tilde{v}_n(k)+k\sum_{i=1}^s b_i \tilde{w}_{n,i}((1-c_i)k), \label{unt}
\end{eqnarray}
where
\begin{eqnarray}
\left\{ \begin{array}{rcl} \dot{\tilde{v}}_n(s)&=& A \tilde{v}_n(s), \\ \tilde{v}_n(0)&=&u_n, \\ \partial \tilde{v}_n(s)&=&\partial [u(t_n)+s A u(t_n)], \end{array}\right. \quad \left\{ \begin{array}{rcl} \dot{\tilde{w}}_{n,j}(s)&=& A \tilde{w}_{n,j}(s), \\ \tilde{w}_{n,j}(0)&=&f(t_n+c_jk,K_{n,j}), \\ \partial \tilde{w}_{n,j}(s)&=&\partial f(t_n,u(t_n)). \end{array}\right. \label{vwunt}
\end{eqnarray}
After discretizing (\ref{vwetapast}) and (\ref{vwunt}) in space, the following systems arise when starting from the discrete numerical approximation $U_{n,h}$ at the previous step, and denoting by $K_{n,h,j}$ to the the discretized stages,
\begin{eqnarray}
&&\left\{ \begin{array}{rcl} \dot{V}_{n,h}(s)&=& A_{h,0} V_{n,h}(s)+C_h \partial u(t_n), \\ V_{n,h}(0)&=&U_{n,h},  \end{array}\right.  \left\{ \begin{array}{rcl} \dot{W}_{n,h,j}(s)&=& A_{h,0} W_{n,h,j}(s), \\ W_{n,h,j}(0)&=&f(t_n+c_jk,K_{n,h,j}).  \end{array}\right. \nonumber \\
&&\left\{ \begin{array}{rcl} \dot{\tilde{V}}_{n,h}(s)&=& A_{h,0} \tilde{V}_{n,h}(s)+C_h [\partial u(t_n)+s A u(t_n)], \\ \tilde{V}_{n,h}(0)&=&U_{n,h},  \end{array}\right.  \nonumber \\
&&\left\{ \begin{array}{rcl} \dot{\tilde{W}}_{n,h,j}(s)&=& A_{h,0} \tilde{W}_{n,h,j}(s)+C_h \partial f(t_n,u(t_n)), \\ \tilde{W}_{n,h,j}(0)&=&f(t_n+c_jk,K_{n,h,j}).  \end{array}\right. \label{discstages}
\end{eqnarray}
By using the variation-of-constants formula and the definition of $\varphi_j$ in (\ref{varphi}), we  have
\begin{eqnarray}
V_{n,h}(s)&=&e^{s A_{h,0}}U_{n,h}+\int_0^s e^{(s-\tau)A_{h,0}} C_h \partial u(t_n)d \tau=e^{s A_{h,0}}U_{n,h}+s \varphi_1(s A_{h,0}) C_h \partial u(t_n), \nonumber \\
\tilde{V}_{n,h}(s)&=&e^{s A_{h,0}}U_{n,h}+\int_0^s e^{(s-\tau)A_{h,0}} C_h [\partial u(t_n)+\tau A u(t_n)]d \tau \nonumber \\
&=&e^{s A_{h,0}}U_{n,h}+s \varphi_1(s A_{h,0}) C_h \partial u(t_n)+s^2 \varphi_2(s A_{h,0}) C_h \partial A u(t_n) , \nonumber \\
\tilde{W}_{n,h,j}(s)&=&e^{s A_{h,0}}f(t_n+c_j k, K_{n,h,j}) +s \varphi_1(s A_{h,0}) C_h \partial f(t_n,u(t_n)). \nonumber
\end{eqnarray}
Therefore, considering (\ref{etapast}) and (\ref{unt}), the full discretized numerical solution after one step  is calculated recursively through the following formulas
\begin{eqnarray}
K_{n,h,i}&=&e^{c_i k A_{h,0}}U_{n,h}+c_i k \varphi_1(c_i k A_{h,0}) C_h \partial u(t_n)+k\sum_{j=1}^{i-1} a_{ij} e^{(c_i-c_j)k A_{h,0}}f(t_n+c_j k, K_{n,h,j}), \nonumber \\
U_{n+1,h}&=&e^{k A_{h,0}}U_{n,h}+k \varphi_1(k A_{h,0}) C_h \partial u(t_n)+k^2 \varphi_2(k A_{h,0}) C_h \partial A u(t_n) \nonumber \\
&&+k\sum_{i=1}^s b_i  \bigg[e^{(1-c_i)k A_{h,0}}f(t_n+c_i k, K_{n,h,i}) \nonumber \\
&& \hspace{2cm} +(1-c_i)k \varphi_1((1-c_i) kA_{h,0}) C_h \partial f(t_n,u(t_n))\bigg]. \label{fd2}
\end{eqnarray}
\begin{remark} Notice that the three terms on the boundary $\partial u(t_n)$, $\partial A u(t_n)$ and $\partial f(t_n,u(t_n))$, are necessary to consider this approximation. However, as
$$\partial u(t_n)=g(t_n), \quad \partial A u(t_n)=\dot{g}(t_n)-\partial f(t_n,u(t_n)),$$
all reduces to calculate $\partial f(t_n,u(t_n))$. In the same way as it was stated in \cite{acrnl,EO2}, with Dirichlet boundary conditions, when the nonlinear term is given by
(\ref{nonlinearterm}), that term can be calculated exactly as
$$\partial f(t_n,u(t_n))=\phi( g(t_n)) + \partial h(t_n).$$
 With Neumann or Robin boundary conditions
\begin{eqnarray}
\partial u(t)=\alpha u(t)|_{\partial \Omega}+\beta \partial_n u(t)|_{\partial \Omega}=g(t), \quad \beta \neq 0,
\label{neurob}
\end{eqnarray}
as
\begin{eqnarray}
\partial f(t_n,u(t_n))= \alpha[\phi(u(t_n)|_{\partial \Omega})+h(t_n)|_{\partial \Omega}]
+\beta [ \phi'(u(t_n)|_{\partial \Omega})\partial_n
u(t_n)|_{\partial \Omega}+\partial_n h(t_n)|_{\partial \Omega}],
\nonumber
\end{eqnarray}
$u(t_n)|_{\partial \Omega}$ can be substituted by the numerical approximation which the space discretization of the problem necessarily gives in this case and $\partial_n
u(t_n)|_{\partial \Omega}$ by the result of solving from (\ref{neurob}). In any case, the error which comes from the approximation of the boundary terms in (\ref{fd2}) is given in Table \ref{tabla}.
\label{rembc}
\end{remark}
\begin{remark}
We also notice that, when $\partial u(t)= \partial A u(t)=0$, because of (\ref{laibvp}), it necessarily happens that $\partial f(t,u(t))=0$. Besides, in this case, (\ref{fd2}) is equivalent to (\ref{lawsonca}). In this way, through Theorems \ref{teorsalocal2} and \ref{teorsalocalfd2}, we will be implicitly proving local order $2$ for the classical approach in such a case.
\end{remark}
\subsubsection{Local error of the time semidiscretization}
In order to define the local error of the time semidiscretization, we consider
\begin{eqnarray}
\bar{K}_{n,i}&=&\bar{v}_n(c_i k)+k\sum_{j=1}^{i-1} a_{ij} \bar{w}_{n,j}((c_i-c_j)k), \quad i=1,\dots,s,
\nonumber \\
\bar{u}_{n+1}&=&\bar{\tilde{v}}_n(k)+k\sum_{i=1}^s b_i \bar{\tilde{w}}_{n,i}((1-c_i)k),
\label{ub}
\end{eqnarray}
where $\bar{v}_n$ and $\bar{\tilde{v}}_n$ satisfy the same equation and boundary conditions as $v_n$ and $\tilde{v}_n$, but starting from $u(t_n)$ instead of $u_n$. The same happens with $\bar{w}_{n,i}$, $\bar{\tilde{w}}_{n,i}$ and $w_{n,i}$, $\tilde{w}_{n,i}$ with the difference that the initial condition is now $f(t_n+c_i k, \bar{K}_{n,i})$ instead of $f(t_n+c_i k, K_{n,i})$. Then, for the local error $\rho_n=\bar{u}_{n+1}-u(t_{n+1})$, we have the following result:
\begin{theorem}
\label{teorsalocal2}
Under hypotheses (A1)-(A5) and (H1)-(H3), assuming also that, for every $t\in [0,T]$, $f(t,u(t))\in D(A)$, and
\begin{eqnarray}
Af(t,u(t))\in C([0,T],X), \quad u\in C([0,T], D(A^2))\cap C^2([0,T],X),
\label{regloc2}
\end{eqnarray}
it follows that $\rho_n=O(k^2)$. Moreover, if $f\in C^2([0,T]\times X,X), u\in C^3([0,T],X)$, there exists a constant $C$ such that the following bound holds
\begin{eqnarray}
\|A_0^{-1} f_u(t,u(t))A_0 w\|\le C, \mbox{ for every } w\in D(A_0) \mbox{ and } t\in [0,T], \label{cspc}
\end{eqnarray}
and the Runge-Kutta tableau corresponds to a method of classical order $\ge 2$, it follows that $A_{0}^{-1} \rho_n=O(k^3)$.
\end{theorem}
\begin{proof} Using Lemma 3.1 in \cite{acr2},
\begin{eqnarray}
\bar{v}_n(s)&=&u(t_n)+s \varphi_1(s A_0) A u(t_n), \label{lemaartlin} \\
\bar{\tilde{v}}(s)&=&u(t_n)+s A u(t_n)+s^2 \varphi_2(s A_0) A^2 u(t_n), \nonumber \\
\bar{w}_{n,i}(s)&=& e^{s A_0} f(t_n+c_i k, \bar{K}_{n,i}), \nonumber \\
\bar{\tilde{w}}_{n,i}(s)&=& e^{s A_0}[f(t_n+c_i k, \bar{K}_{n,i})-f(t_n,u(t_n))]+f(t_n,u(t_n))+s \varphi_1(s A_0) A f(t_n,u(t_n)). \nonumber
\end{eqnarray}
Then,
\begin{eqnarray}
\bar{K}_{n,i}&=& u(t_n)+ c_i k \varphi_1(c_i k A_0) A u(t_n)+k\sum_{j=1}^{i-1} a_{ij} e^{(c_i-c_j)kA_0} f(t_n+c_j k, \bar{K}_{n,j})\label{Kb} \\
&=&u(t_n)+O(k), \qquad i=1,\dots,s, \label{acotet} \\
\bar{u}_{n+1}&=&u(t_n)+k A u(t_n)+ k^2 \varphi_2(k A_0) A^2 u(t_n) \nonumber \\
&&+k \sum_{i=1}^s b_i \bigg[e^{(1-c_i)k A_0}[f(t_n+c_i k,\bar{K}_{n,i})-f(t_n,u(t_n))]+f(t_n,u(t_n))\nonumber \\
&& \hspace{2cm} +(1-c_i)k \varphi_1((1-c_i)k A_0) A f(t_n,u(t_n))\bigg] \label{bu2} \\
&=& u(t_n)+k[Au(t_n)+f(t_n,u(t_n))]+O(k^2)= u(t_n)+k \dot{u}(t_n)+O(k^2), \nonumber
\end{eqnarray}
where, for the last line, (\ref{regloc2}), (A4) together with (\ref{acotet}) and the first condition of (\ref{cscond}) have been used. From this, the first result on the local error follows.

As for the second result, looking at the term in $k^2$ in $\bar{u}_{n+1}$ and using that $u\in C^3([0,T],X)$, we can notice that
\begin{eqnarray}
\lefteqn{A_0^{-1} \rho_{n+1}= k^3  (k A_0)^{-1}(\varphi_2(k A_0)-\frac{1}{2} I) A^2 u(t_n)+\frac{k^2}{2} A_0^{-1} A^2 u(t_n)} \nonumber \\
&&+k^2 \sum_{i=1}^s b_i(1-c_i) ((1-c_i)k A_0)^{-1}[e^{(1-c_i)k A_0}-I][f(t_n+c_i k,\bar{K}_{n,i})-f(t_n,u(t_n))]\nonumber \\
&&+k \sum_{i=1}^s b_i A_0^{-1} [f(t_n+c_i k,\bar{K}_{n,i})-f(t_n,u(t_n))] \label{forloc} \\
&&+k^3 \sum_{i=1}^s b_i(1-c_i)^2 ((1-c_i)k A_0)^{-1}(\varphi_1((1-c_i)k A_0)- I)A f(t_n,u(t_n)) \nonumber \\
&&+k^2 \sum_{i=1}^s b_i(1-c_i)A_0^{-1} A f(t_n,u(t_n))-\frac{k^2}{2} \ddot{u}(t_n)+O(k^3). \nonumber
\end{eqnarray}
Using (\ref{recurf}), (\ref{acotet}), (A4) and (\ref{regloc2}),  the first, third and fifth terms are $O(k^3)$. As for the fourth one, considering (\ref{Kb}), it can be written as
\begin{eqnarray}
&&k^2 \sum_{i=1}^s b_i A_0^{-1} f_u(t_n,u(t_n))[c_i \varphi_1(c_i k A_0) A u(t_n)+\sum_{j=1}^{i-1} a_{ij} e^{(c_i-c_j)k A_0} f(t_n+c_j k, \bar{K}_{n,j})] \nonumber \\
&& \hspace{0.5cm}+k^2 \sum_{i=1}^s b_i c_i A_0^{-1} f_t(t_n,u(t_n))+O(k^3) \nonumber \\
&&=k^3 \sum_{i=1}^s b_i A_0^{-1} f_u(t_n,u(t_n))A_0 c_i (k A_0)^{-1}[\varphi_1(c_i k A_0) -I] A u(t_n)\nonumber \\
&& \hspace{0.5cm}+k^2 (\sum_{i=1}^s b_i c_i) A_0^{-1} f_u(t_n,u(t_n))A u(t_n)\nonumber \\
&& \hspace{0.5cm}+k^3 \sum_{i=1}^s b_i\sum_{j=1}^{i-1} a_{ij} A_0^{-1} f_u(t_n,u(t_n))A_0 (k A_0)^{-1}[e^{(c_i-c_j)k A_0}-I] f(t_n, u(t_n))\nonumber \\
&& \hspace{0.5cm}+k^2 \sum_{i=1}^s b_i \sum_{j=1}^{i-1} a_{ij} A_0^{-1} f_u(t_n,u(t_n))f(t_n,u(t_n)) \nonumber \\
&& \hspace{0.5cm}+k^2 \sum_{i=1}^s b_i c_i A_0^{-1} f_t(t_n,u(t_n))+O(k^3) \nonumber \\
&&=\frac{k^2}{2} A_0^{-1}[f_u(t_n,u(t_n))\dot{u}(t_n)+f_t(t_n,u(t_n))]+O(k^3) \nonumber,
\end{eqnarray}
where, for the last equality, we have used (\ref{recurf}) again, (\ref{cspc}), the second condition in (\ref{cscond}) and the fact that $\sum_{i=1}^s b_i c_i=1/2$ due to the second order of the Butcher tableau. Inserting this in (\ref{forloc}) and simplifying notation,
$$A_0^{-1} \rho_{n+1}=\frac{k^2}{2}A_0^{-1} [A^2 u+f_u\dot{u}+f_t+A f-\ddot{u}]+O(k^3)=O(k^3),$$
where the differentiation of (\ref{laibvp}) with respect to time shows that term in bracket in the previous expression vanishes.
\end{proof}

\subsubsection{Local error of the full discretization}
We define the local error of the full discretization as
$\rho_{n+1,h}=\bar{U}_{n+1,h}-P_h u(t_{n+1})$, where $\bar{U}_{n+1,h}$ is defined as $U_{n+1,h}$ but starting from $P_h u(t_n)$. More precisely, in a similar way to the derivation of (\ref{fd2}), $\bar{U}_{n+1,h}$ is defined through some stages $\bar{K}_{n,i,h}$ in the following way:
\begin{eqnarray}
\bar{K}_{n,h,i}&=&\bar{V}_{n,h}(c_i k)+k \sum_{j=1}^{i-1} a_{ij} \bar{W}_{n,h,j}((c_i-c_j)k), \quad i=1,\dots,s,\nonumber \\
\bar{U}_{n+1,h}&=& \bar{\tilde{V}}_{n,h}(k)+k\sum_{i=1}^s b_i \bar{\tilde{W}}_{n,h,i}((1-c_i)k), \nonumber
\end{eqnarray}
where $\bar{V}_{n,h}$, $\bar{W}_{n,j,h}$, $\bar{\tilde{V}}_{n,h}$ and $\bar{\tilde{W}}_{n,j,h}$ are defined as in (\ref{discstages}), but changing $U_{n,h}$ by $P_h u(t_n)$ and $K_{n,h,i}$ by $\bar{K}_{n,h,i}$. We then have the following result for the full discretized local error:
\begin{theorem}Under the same hypotheses of the first part of Theorem \ref{teorsalocal2}, and assuming also that, for $t\in [0,T]$,
\begin{eqnarray}
A^l u(t)\in Z, \quad l=0,1,2, \qquad A f(t,u(t))\in Z,
\label{reg2}
\end{eqnarray}
it happens that $\rho_{n,h}=O(k^2+k \varepsilon_h)$ where the constant in Landau notation is independent of $k$ and $h$. Moreover, under the additional hypotheses of the second part of Theorem \ref{teorsalocal2}, together with the following condition which is related to (\ref{cspc}),
\begin{eqnarray}
\|A_{h,0}^{-1} f_u(t,P_h u(t))A_{h,0}\|_h \le C, \quad t\in [0,T], \quad h\le h_0,
\label{cspch}
\end{eqnarray}
it follows that
$A_{h,0}^{-1} \rho_{n,h}=O(k^3+k \eta_h+k^2 \epsilon_h)$.
\label{teorsalocalfd2}
\end{theorem}
\begin{proof}
For the first part of the theorem, as
\begin{eqnarray}
\rho_{n+1,h}&=&\bar{U}_{n+1,h}-P_h u(t_{n+1})=(\bar{U}_{n+1,h}-P_h \bar{u}_{n+1})+P_h (\bar{u}_{n+1}-u(t_{n+1}))\nonumber \\
&=&(\bar{U}_{n+1,h}-P_h \bar{u}_{n+1})+P_h \rho_{n+1}, \label{decompl}
\end{eqnarray}
because of Theorem \ref{teorsalocal2}, it suffices to prove that $\bar{U}_{n+1,h}-P_h \bar{u}_{n+1}=O(k \varepsilon_h)$. Considering the definition of $\bar{u}_{n+1}$ (\ref{ub}),
\begin{eqnarray}
\bar{U}_{n+1,h}-P_h \bar{u}_{n+1}=\bar{\tilde{V}}_{n,h}(k)-P_h \bar{\tilde{v}}_n(k)+k \sum_{i=1}^s b_i [\bar{\tilde{W}}_{n,i,h}-P_h \bar{\tilde{w}}_{n,i}((1-c_i) k)],
\label{difUb}
\end{eqnarray}
where $\bar{\tilde{v}}_n$ and $\bar{\tilde{w}}_{n,i}$ satisfy (\ref{vwunt}) with $u_n$ substituted by $u(t_n)$ and $K_{n,j}$ substituted by $\bar{K}_{n,j}$.
Then,
\begin{eqnarray}
\dot{\bar{\tilde{V}}}_{n,h}(s)-P_h \dot{\bar{\tilde{v}}}_n(s)&=&A_{h,0} \bar{\tilde{V}}_{n,h}(s)-A_{h,0} R_h \bar{\tilde{v}}_n(s) \nonumber \\
&=& A_{h,0}(\bar{\tilde{V}}_{n,h}(s)-P_h \bar{\tilde{v}}_n(s))+ A_{h,0}(P_h-R_h)\bar{\tilde{v}}_n(s), \nonumber \\
\bar{\tilde{V}}_{n,h}(0)-P_h \bar{\tilde{v}}_n(0)&=&0. \nonumber
\end{eqnarray}
Because of (\ref{lemaartlin}), (H2a) and the hypotheses of regularity (\ref{reg2}), $\bar{\tilde{v}}_n(s)\in Z$, which implies, using (\ref{consistency}) and the variation-of-constants formula, that
\begin{eqnarray}
\bar{\tilde{V}}_{n,h}(k)-P_h \bar{\tilde{v}}_n(k)=\int_0^k e^{(k-s)A_{h,0}} A_{h,0}(P_h-R_h)\bar{\tilde{v}}_n(s)=O(k \varepsilon_h).
\label{pvV}
\end{eqnarray}
In a similar way, using also (\ref{fandph}),
\begin{eqnarray}
\dot{\bar{\tilde{W}}}_{n,h,i}(s)-P_h \dot{\bar{\tilde{w}}}_{n,i}(s)&=&A_{h,0} \bar{\tilde{W}}_{n,i,h}(s)-A_{h,0} R_h \bar{\tilde{w}}_{n,i}(s) \nonumber \\
&=& A_{h,0}(\bar{\tilde{W}}_{n,h,i}(s)-P_h \bar{\tilde{w}}_{n,i}(s))+ A_{h,0}(P_h-R_h)\bar{\tilde{w}}_{n,i}(s), \nonumber \\
\bar{\tilde{W}}_{n,h,i}(0)-P_h \bar{\tilde{w}}_{n,i}(0)&=&f(t_n+c_i k, \bar{K}_{n,h,i})-f(t_n+c_i k, P_h \bar{K}_{n,i}), \nonumber
\end{eqnarray}
where $\bar{\tilde{w}}_{n,i}(s)$ again belongs to $Z$ because of the definition of $\bar{K}_{n,i}$ (\ref{ub}), (\ref{lemaartlin}), (H2a) and the conditions of regularity (\ref{reg2}). Proceeding as before,
\begin{eqnarray}
\lefteqn{\bar{\tilde{W}}_{n,h,i}((1-c_i)k)-P_h \bar{\tilde{w}}_{n,i}((1-c_i)k)} \nonumber \\
&=&e^{ (1-c_i) k A_{h,0}}[f(t_n+c_i k, \bar{K}_{n,h,i})-f(t_n+c_i k, P_h \bar{K}_{n,i})]+O(k \varepsilon_h),
\label{Pw-W}
\end{eqnarray}
where
\begin{eqnarray}
\bar{K}_{n,h,i}-P_h \bar{K}_{n,i}=\bar{V}_{n,h}(c_i k)-P_h \bar{v}_n(c_i k)+k \sum_{j=1}^{i-1} a_{ij}[\bar{W}_{n,h,j}((c_i-c_j)k)-P_h \bar{w}_{n,j}((c_i-c_j)k)].
\label{Kb-PK}
\end{eqnarray}
Now, with similar arguments as those for deducing (\ref{pvV}) and (\ref{Pw-W}),
\begin{eqnarray}
\bar{V}_{n,h}(c_i k)-P_h \bar{v}_n (c_i k)
&=&O(k \varepsilon_h)
\nonumber \\
\bar{W}_{n,h,j}((c_i-c_j)k)-P_h \bar{w}_{n,j}((c_i-c_j)k)&=& e^{(c_i-c_j)k A_{h,0}}[f(t_n+c_j k, \bar{K}_{n,h,j})-f(t_n+c_j k, P_h \bar{K}_{n,j})].
\nonumber
\end{eqnarray}
Therefore, writing this in (\ref{Kb-PK}), it is inductively proved that
\begin{eqnarray}
\bar{K}_{n,h,i}-P_h \bar{K}_{n,i}=O(k \varepsilon_h),
\label{Kb-PK2}
\end{eqnarray}
with which the first result of the theorem is proved considering that in (\ref{Pw-W}) and then in (\ref{difUb}) together with (\ref{pvV}).

As for the second part of the theorem, notice that from (\ref{decompl}),
\begin{eqnarray}
A_{h,0}^{-1} \rho_{n,h}=A_{h,0}^{-1}(\bar{U}_{n+1,h}-P_h \bar{u}_{n+1})+A_{h,0}^{-1} P_h \rho_{n+1}.
\label{Airho}
\end{eqnarray}
For the first term in (\ref{Airho}), we then notice that
\begin{eqnarray}
A_{h,0}^{-1}(\bar{\tilde{V}}_{n,h}(k)-P_h \bar{\tilde{v}}_n(k))=O(k \eta_h),
\label{AinvV}
\end{eqnarray}
considering (\ref{pvV}) multiplied by $A_{h,0}^{-1}$ and the second formula of (\ref{consistency}). On the other hand, considering (\ref{Pw-W}),
\begin{eqnarray}
\lefteqn{A_{h,0}^{-1}[\bar{\tilde{W}}_{n,h,i}((1-c_i)k)-P_h \bar{\tilde{w}}_{n,i}((1-c_i)k)]} \nonumber \\
&=&A_{h,0}^{-1} e^{(1-c_i)k A_{h,0}}[f(t_n+c_i k, \bar{K}_{n,h,i})-f(t_n+c_i k, P_h \bar{K}_{n,i})]+O(k \eta_h),
\label{AinvW}
\end{eqnarray}
but
\begin{eqnarray}
\lefteqn{A_{h,0}^{-1} e^{(1-c_i)k A_{h,0}}[f(t_n+c_i k, \bar{K}_{n,h,i})-f(t_n+c_i k, P_h \bar{K}_{n,i})]} \nonumber \\
&=&e^{(1-c_i)k A_{h,0}} A_{h,0}^{-1} f_u(t_n,P_h u(t_n))(\bar{K}_{n,h,i}-P_h \bar{K}_{n,i})+O(k^2 \varepsilon_h), \nonumber
\end{eqnarray}
taking (\ref{acotet}) and (\ref{Kb-PK2}) into account. Moreover, by similar arguments as above,
\begin{eqnarray}
A_{h,0}^{-1} (\bar{K}_{n,h,i}-P_h \bar{K}_{n,i})=O(k \eta_h). \nonumber
\end{eqnarray}
Therefore, using (\ref{cspch}),
\begin{eqnarray}
A_{h,0}^{-1} e^{(1-c_i)k A_{h,0}}[f(t_n+c_i k, \bar{K}_{n,h,i})-f(t_n+c_i k, P_h \bar{K}_{n,i})]=O(k \eta_h+k^2 \varepsilon_h), \nonumber
\end{eqnarray}
and so, inserting this in (\ref{AinvW}) and using also (\ref{AinvV}),
\begin{eqnarray}
A_{h,0}^{-1}[\bar{U}_{n+1,h}-P_h \bar{u}_{n+1}]=O(k \eta_h+k^2 \varepsilon_h).
\label{AinvU}
\end{eqnarray}
For the second term in (\ref{Airho},  we notice that, as $\partial A_0^{-1} \rho_{n+1}=0$, applying (\ref{rh}), it follows that $A_{h,0} R_h (A_0^{-1} \rho_{n+1})=P_h \rho_{n+1}$, from what
\begin{eqnarray}
A_{h,0}^{-1} P_h \rho_{n+1}=R_h  (A_0^{-1} \rho_{n+1})=P_h (A_0^{-1} \rho_{n+1})+(R_h -P_h) (A_0^{-1} \rho_{n+1}).
\label{Ah0iP}
\end{eqnarray}
Then, $P_h A_0^{-1} \rho_{n+1}$ here is $O(k^3)$ because of the second part of Theorem \ref{teorsalocal2}. Secondly, $(R_h-P_h)A_0^{-1} \rho_{n+1}$ is $O(k \eta_h)$ because $A_0^{-1} \rho_{n+1}$ belongs to $Z$ due to (\ref{bu2}), (\ref{Kb}), the conditions (\ref{reg2}) and (H2a). Moreover,  $\|A_0^{-1} \rho_{n+1}\|_Z=O(k)$ because $u(t_{n+1})=u(t_n)+k \dot{u}(t_n^*)$ for $t_n^* \in [t_n,t_{n+1}]$, with $\dot{u}(t_n^*)=A u(t_n^*)+f(t_n^*, u(t_n^*))\in Z$.
Therefore, $A_{h,0}^{-1} P_h \rho_{n+1}=O(k^3+k \eta_h)$ and the result follows considering also (\ref{AinvU}) in (\ref{Airho}).
\end{proof}

\subsubsection{Global error of the full discretization}

From the first part of Theorem \ref{teorsalocalfd2}, again the classical argument would lead to global error $e_{n,h}=U_{n,h}-P_h u(t_n)=O(k+\varepsilon_h)$. However, for parabolic problems, for which (\ref{spp}) is expected to hold, using the second part of the same theorem, a summation-by-parts argument leads to second order in time, as the following theorem states. The theorem is valid for both Dirichlet and Neumann/Robin boundary conditions, in spite of the fact that,  for the latter, $\partial f(t,u(t))$ must be approximated through the numerical solution itself, as explained in Remark \ref{rembc}.
\begin{theorem}
Under hypotheses of Theorem \ref{teorsalocalfd2}, but assuming also (\ref{spp}) and that
\begin{eqnarray}
&&\hspace{-0.5cm}f\in C^3([0,t]\times X, X), \,\, A f(t,u(t))\in C^1([0,T],X),  \,\,  u\in C^1([0,T], D(A^2))\cap C^4([0,T],X), \nonumber \\
&&\hspace{-0.5cm} A^l \dot{u}(t)\in Z, \, l=0,1,2, \quad \frac{d}{dt} A^l f(t,u(t))\in Z, \, l=0,1, \nonumber
\end{eqnarray}
it follows that $e_{n,h}=U_{n,h}-P_h u(t_n)=O(k^2+k \varepsilon_h+\eta_h)$.
\label{teorsaglobfd2}
\end{theorem}
\begin{proof}
Firstly notice that
\begin{eqnarray}
e_{n+1,h}=[U_{n+1,h}-\bar{U}_{n+1,h}]+[\bar{U}_{n+1,h}-P_h u(t_{n+1})]=[U_{n+1,h}-\bar{U}_{n+1,h}]+\rho_{n+1,h}.
\label{decompg}
\end{eqnarray}
Then, using (\ref{fd2}), when considering Dirichlet boundary conditions, in which case $\partial Au(t_n)$ and $\partial f(t_n,u(t_n))$ are calculated exactly in terms of data, as $\bar{U}_{n+1,h}$ is the same as $U_{n+1,h}$ but starting from $P_h u(t_n)$ instead of $U_{n,h}$,
\begin{eqnarray}
\lefteqn{U_{n+1,h}-\bar{U}_{n+1,h}=e^{k A_{h,0}}[U_{n,h}-P_h u(t_n)]}
\nonumber \\
&&+k \sum_{i=1}^s b_i e^{(1-c_i)k A_{h,0}}[f(t_n+c_i k, K_{n,h,i})-f(t_n+c_i k, \bar{K}_{n,h,i})], \label{UUb}
\end{eqnarray}
where, recursively, for $i=1,\dots,s$,
\begin{eqnarray}
\lefteqn{K_{n+1,h,i}-\bar{K}_{n+1,h,i}=e^{c_i k A_{h,0}}[U_{n,h}-P_h u(t_n)]}\nonumber \\
&&+k \sum_{j=1}^{i-1} a_{ij}e^{(c_i-c_j)k A_{h,0}}[f(t_n+c_j k, K_{n,h,j})-f(t_n+c_j k, \bar{K}_{n,h,j})]. \label{KKb}
\end{eqnarray}
In such a way, it is inductively proved that $K_{n+1,h,i}-\bar{K}_{n+1,h,i}=O(e_{n,h})$ and finally, using (H3), (\ref{UUb}) and (\ref{decompg}),
\begin{eqnarray}
e_{n+1,h}=e^{k A_{h,0}} e_{n,h}+O(k e_{n,h})+\rho_{n+1,h},
\label{recurg1}
\end{eqnarray}
from what the result follows from Theorem \ref{teorsalocalfd2} by a summation-by-parts argument and a discrete Gronwall lemma in the same way than the proof of Theorem 22 in \cite{acrnl} for Strang method.

On the other hand, when considering Robin/Neumann boundary conditions, as, according to Remark \ref{rembc},  $\partial f(t_n,u(t_n))$  is just calculated approximately with an error which is $O(e_{n,h})$,
 using (\ref{fd2}) again,
\begin{eqnarray}
\lefteqn{U_{n+1,h}-\bar{U}_{n+1,h}=e^{k A_{h,0}} e_{n,h} + k^2 \varphi_2( k A_{h,0})C_h O(e_{n,h})}\nonumber \\
&&+k \sum_{i=1}^{s} b_i \bigg[e^{(1-c_i)k A_{h,0}}[f(t_n+c_i k, K_{n,h,i})-f(t_n+c_i k, \bar{K}_{n,h,i})]
\nonumber \\
&& \hspace{2cm}+ (1-c_i)k \varphi_1((1-c_i)k A_{h,0}) C_h O(e_{n,h})\bigg], \nonumber
\end{eqnarray}
where $K_{n+1,h,i}-\bar{K}_{n+1,h,i}$ is the same as in (\ref{KKb}) because $\partial u(t_n)$ is given exactly in terms of data with this type of boundary conditions. Then, using (\ref{recurf}) and (H3),
\begin{eqnarray}
\lefteqn{\hspace{-1cm}U_{n+1,h}-\bar{U}_{n+1,h}=e^{ k A_{h,0}} e_{n,h} + k [\varphi_1( k A_{h,0})-I] A_{h,0}^{-1} C_h O(e_{n,h})}\nonumber \\
&&+k \sum_{i=1}^{s} b_i \bigg[e^{(1-c_i)k A_{h,0}}O(e_{n,h})
+ [e^{(1-c_i)k A_{h,0}}-I]A_{h,0}^{-1} C_h O(e_{n,h}) \bigg]. \nonumber
\end{eqnarray}
Using now (H2c), it follows that $U_{n+1,h}-\bar{U}_{n+1,h}=e^{ k A_{h,0}} e_{n,h}+O(k e_{n,h})$, from what (\ref{recurg1}) applies again and the result follows in the same way as above.
\end{proof}

\subsection{Searching for local order 3}
For the stages, we again consider (\ref{etapast}), but where now $v_n$ and $w_{n.,j}$ are those in (\ref{vwunt}) instead of those in (\ref{vwetapast}). On the other hand, $u_{n+1}$ is calculated through (\ref{unt}) where now $\tilde{v}_n$, $\tilde{w}_{n,j} (j=1,\dots,s)$ satisfy
\begin{eqnarray}
&&\left\{ \begin{array}{rcl} \dot{\tilde{v}}_n(s)&=& A \tilde{v}_n(s), \\ \tilde{v}_n(0)&=&u_n, \\ \partial \tilde{v}_n(s)&=&\partial [u(t_n)+s A u(t_n)+\frac{s^2}{2} A^2 u(t_n)], \end{array}\right.
 \nonumber \\
 &&\left\{ \begin{array}{rcl} \dot{\tilde{w}}_{n,j}(s)&=& A \tilde{w}_{n,j}(s), \\ \tilde{w}_{n,j}(0)&=&f(t_n+c_jk,K_{n,j}), \\ \partial \tilde{w}_{n,j}(s)&=&\partial [f(t_n+c_j k,u(t_n)+c_j k \dot{u}(t_n))+s A f(t_n,u(t_n))]. \end{array}\right. \label{timelocal3}
\end{eqnarray}
After discretizing in space and using the variation-of-constants formula as in Subsection 4.1, the full discretized numerical solution after one step is given by
\begin{eqnarray}
K_{n,h,i}&=&e^{c_i k A_{h,0}}U_{n,h}+c_i k \varphi_1(c_i k A_{h,0}) C_h \partial u(t_n)+c_i^2 k^2 \varphi_2(c_i k A_{h,0}) C_h \partial A u(t_n) \nonumber \\
&&+
k\sum_{j=1}^{i-1} a_{ij}\bigg[ e^{(c_i-c_j) k A_{h,0}}f(t_n+c_j k, K_{n,h,j}) \nonumber \\
&&\hspace{2cm}+(c_i-c_j) k \varphi_1((c_i-c_j)k A_{h,0})C_h \partial f(t_n,u(t_n))\bigg], \quad i=1,\dots,s,\nonumber \\
U_{n+1,h}&=&e^{k A_{h,0}}U_{n,h}+\sum_{l=1}^3 k^l \varphi_l(k A_{h,0}) C_h \partial A^{l-1} u(t_n) \nonumber \\
&&+k\sum_{i=1}^s b_i  \bigg[e^{(1-c_i)k A_{h,0}}f(t_n+c_i k, K_{n,h,i}) \nonumber \\
&& \hspace{2cm} +(1-c_i)k \varphi_1((1-c_i) kA_{h,0}) C_h \partial f(t_n+c_i k,u(t_n)+c_i k \dot{u}(t_n)) \nonumber \\
&& \hspace{2cm} +(1-c_i)^2 k^2 \varphi_2((1-c_i) kA_{h,0}) C_h \partial A f(t_n,u(t_n))
\bigg].\label{fd3}
\end{eqnarray}

\begin{remark}
Notice that, apart from the terms on the boundary which were already necessary to achieve local order $2$ and which can be calculated according to Remark \ref{rembc}, now we also need $\partial A^2 u(t_n)$, $\partial f(t_n+c_i k,u(t_n)+c_i k \dot{u}(t_n))$ and $\partial A f(t_n,u(t_n))$.

With Dirichlet boundary conditions and $f$ like in (\ref{nonlinearterm}), it follows that
\begin{eqnarray}
\partial f(t_n+c_i k,u(t_n)+c_i k \dot{u}(t_n))&=&\phi (g(t_n)+c_i k \dot{g}(t_n))+ \partial h(t_n+c_i k),
\nonumber \\
\partial A^2 u(t_n)&=& \ddot{g}(t_n)-\partial \dot{h}(t_n)-\phi'(g(t_n))\dot{g}(t_n)-\partial A f(t_n,u(t_n)), \nonumber
\end{eqnarray}
and the only term which cannot be calculated exactly in terms of data is $\partial A f(t_n,u(t_n))$. However, that can be approximated recurring to numerical differentiation. For example, in one dimension and assuming that $A$ is the second spatial derivative,
$$A f(t_n,u(t_n))=\phi''(u(t_n)) u_x(t_n)^2+\phi'(u(t_n))u_{xx}(t_n)+h_{xx}(t_n),$$
 from what
\begin{eqnarray}
\partial Af(t_n,u(t_n))\approx \phi''(g(t_n)) \hat{u}_x(t_n)^2|_{\partial \Omega}+\phi'(g(t_n))(\dot{g}(t_n)-\phi(g(t_n))-\partial h(t_n))+\partial h_{xx}(t_n), \nonumber
\end{eqnarray}
where $\hat{u}_x(t_n)|_{\partial \Omega}$ is the result of applying numerical differentiation to approximate $u_x(t_n)$ on the boundary. For that, both the exact values at the boundary and the approximated values at the interior of the domain given by the numerical approximation must be used. As a result,
$
\hat{u}_x(t_n)-u_x(t_n)=O(\nu_h+\frac{e_{n,h}}{h}),
$
where $\nu_h$ decreases with $h$ and comes from the error of the numerical differentiation if the exact values of the solutions were used. The second term $\frac{e_{n,h}}{h}$ comes from the fact that the values at the interior are just the approximations which are given by the numerical solution and to the necessity of dividing by $h$ when approximating a first derivative in space. For a general operator $A$, we will assume that the error when approximating both $\partial A^2 u(t_n)$ and $\partial A f(t_n,u(t_n))$ is
as specified in Table \ref{tabla}
for some real value $\gamma$, where $\nu_h$ comes from the numerical approximation of the corresponding derivatives in space if the exact values had been taken.

As for Robin/Neumann boundary conditions (\ref{neurob}), we notice that
\begin{eqnarray}
\lefteqn{\partial f(t_n+c_i k,u(t_n)+c_i k \dot{u}(t_n))=\alpha [\phi (u(t_n)|_{\partial \Omega}+c_i k \dot{u}(t_n)|_{\partial \Omega})+  h(t_n+c_i k)|_{\partial \Omega}]} \nonumber \\
&&+\beta[\phi' (u(t_n)|_{\partial \Omega}+c_i k \dot{u}(t_n)|_{\partial \Omega})(\partial_n u(t_n)|_{\partial \Omega}+c_i k \partial_n \dot{u}(t_n)|_{\partial \Omega})+ \partial_n h(t_n+c_i k)|_{\partial \Omega}].
 \nonumber
 \end{eqnarray}
Here $u(t_n)|_{\partial \Omega}$ and $\partial_n u(t_n)|_{\partial \Omega}$ are approximated through the numerical solution, as in Remark \ref{rembc},  $\dot{u}(t_n)|_{\partial \Omega}$ is then approximated through numerical differentiation in time from the approximated values at the boundary and $\partial_n \dot{u}(t_n)|_{\partial \Omega}$ is then solved from the differentiation in time of (\ref{neurob}). In such a way, if the error coming from the numerical differentiation in time from the exact values is $O(\mu_{k,1})$, with $\mu_{k,1}$ decreasing when $k$ decreases, it happens that the error when approximating $\partial f(t_n+c_i k,u(t_n)+c_i k \dot{u}(t_n))$ is $O(k \mu_{k,1}+e_{n,h})$, with the same argument as before for space numerical differentiation, and taking now into account the factor $k$ which is multiplying the corresponding derivative. As for $\partial A^2 u(t_n)$, using (\ref{laibvp}), $ \partial A^2 u=\ddot{g}-\partial[\dot{h}+\phi'(u)\dot{u}]-\partial (Af)$. Moreover,
\begin{eqnarray}
\partial[\dot{h}+\phi'(u)\dot{u}]=\alpha[\dot{h}|_{\partial \Omega}+\phi'(u|_{\partial \Omega})\dot{u}|_{\partial \Omega}]+\beta [\partial_n\dot{h}|_{\partial \Omega}+\phi''(u|_{\partial \Omega})\partial_n u|_{\partial \Omega}\dot{u}|_{\partial \Omega}+\phi'(u|_{\partial \Omega})\partial_n \dot{u}|_{\partial \Omega}],
\nonumber
\end{eqnarray}
and it is then necessary to approximate $u|_{\partial \Omega}$, $\partial_n u|_{\partial \Omega}$, $\dot{u}|_{\partial \Omega}$ and $\partial_n \dot{u}|_{\partial \Omega}$, as before.
On the other hand, when $A$ is the second derivative in one dimension,
\begin{eqnarray}
\partial (Af)&=&\alpha\bigg[\phi''(u|_{\partial \Omega}) u_x^2|_{\partial \Omega}+\phi'(u|_{\partial \Omega})[\dot{u}|_{\partial \Omega}-\phi(u|_{\partial \Omega})-h|_{\partial \Omega}]+h_{xx}|_{\partial \Omega}\bigg] \nonumber \\
&&+\beta \bigg[\phi'''(u|_{\partial \Omega}) u_x^3|_{\partial \Omega}+3 \phi''(u|_{\partial \Omega}) u_x|_{\partial \Omega}[\dot{u}|_{\partial \Omega}-\phi(u|_{\partial \Omega})-h|_{\partial \Omega}]\nonumber \\
&&\hspace{1cm}+\phi'(u|_{\partial \Omega})[\dot{u}_x|_{\partial \Omega}-\phi'(u|_{\partial \Omega})u_x|_{\partial \Omega}-h_x|_{\partial \Omega}]+h_{xxx}|_{\partial \Omega}\bigg]. \nonumber
\end{eqnarray}
Therefore, in this particular case, approximating  $u|_{\partial \Omega}$, $\dot{u}|_{\partial \Omega}$, $u_x|_{\partial \Omega}$ and $\dot{u}_x|_{\partial \Omega}$ as above, the error which comes from calculating $\partial A^2 u(t_n)$ and $\partial A f(t_n,u(t_n))$ is $O(\mu_{k,1}+\frac{e_{n,h}}{k})$. However, for more general operators, space numerical differentiation may be also needed, and therefore we will assume that, in general, the error coming from the calculation of those boundaries is as specified in Table \ref{tabla}.
\label{rembc3}
\end{remark}

\begin{remark}
We notice that, if $\partial u(t)=\partial A u(t)= \partial A^2 u(t)=0$, from (\ref{laibvp}) it follows that $\partial f(t,u(t))=\partial A f(t,u(t))=0$. In particular, this implies that $\partial f(t_n+c_i k, u(t_n+c_i k))=0$, which differs from $\partial f(t_n+c_i k, u(t_n)+c_i k \dot{u}(t_n))$ in $O(k^2)$. Then, in this case, each step in (\ref{fd3}) differs from the classical approach (\ref{lawsonca}) in $O(k^3)$ since the difference is
\begin{eqnarray}
k\sum_{i=1}^s b_i (1-c_i)k \varphi_1((1-c_i)k A_{h,0})C_h O(k^2)=k \sum_{i=1}^s b_i [e^{(1-c_i)k A_{h,0}}-I]A_{h,0}^{-1}C_h O(k^2), \nonumber
\end{eqnarray}
and, according to (H2c), $A_{h,0}^{-1}C_h$ is uniformly bounded. This justifies, through Theorems \ref{teorsalocal3} and \ref{teorsalocalfd3}, that the local error with the classical approach, under these particular boundary conditions, behaves with order $3$ under the assumptions of those theorems.
\label{remvan3}
\end{remark}

\subsubsection{Local error of the time semidiscretization}

In a similar way to Theorem \ref{teorsalocal2}, we have the following result:
\begin{theorem}
\label{teorsalocal3}
Under hypotheses (A1)-(A5) and (H1)-(H3), assuming also that, for every $t\in [0,T]$, $f(t,u(t))\in D(A^2)$, that, for small enough $\tau$ ($\tau\le \tau_0$), $f(t+\tau,u(t)+\tau \dot{u}(t)),f_t(t+\tau,u(t)+\tau \dot{u}(t)), f_u(t+\tau,u(t)+\tau \dot{u}(t))\dot{u}(t)\in D(A)$ and
\begin{eqnarray}
&& \hspace{-0.5cm} u\in C([0,T], D(A^3))\cap C^3([0,T],X), \nonumber \\
&&\hspace{-0.5cm} f\in C^2([0,T]\times X,X), \nonumber \\
&&\hspace{-0.5cm} A^l f(\cdot,u(\cdot))\in C([0,T],X), \quad l=1,2, \nonumber \\
&& \hspace{-0.5cm} Af_t(t+\tau,u(t)+\tau \dot{u}(t)), A[f_u(t+\tau,u(t)+\tau \dot{u}(t))\dot{u}(t)] \in C([0,T]\times [0,\tau_0],X),
\label{regloc3}
\end{eqnarray}
and that the Runge-Kutta tableau corresponds to a method of classical order $\ge 2$,
it follows that $\rho_n=O(k^3)$. Moreover, if $f\in C^3([0,T]\times X,X), u\in C^4([0,T],X)$, (\ref{cspc}) holds
and the Runge-Kutta tableau corresponds to a method of classical order $\ge 3$, it follows that $A_{0}^{-1} \rho_n=O(k^4)$.
\end{theorem}
\begin{proof} Using Lemma 3.1 in \cite{acr2} again,
\begin{eqnarray}
\bar{v}_n(s)&=&u(t_n)+s A u(t_n)+s^2 \varphi_2(s A_0) A^2 u(t_n), \nonumber \\
\bar{\tilde{v}}_n(s)&=&u(t_n)+s A u(t_n)+\frac{s^2}{2} A^2 u(t_n)+ s^3 \varphi_3(s A_0) A^3 u(t_n), \nonumber \\
\bar{w}_{n,i}(s)&=& e^{s A_0}[f(t_n+c_i k, \bar{K}_{n,i})-f(t_n,u(t_n))]+f(t_n,u(t_n))+s \varphi_1(s A_0) A f(t_n,u(t_n)), \nonumber
\end{eqnarray}
and with a similar argument to that used in that lemma,
\begin{eqnarray}
\bar{\tilde{w}}_{n,i}(s)&=& e^{s A_0}[f(t_n+c_i k, \bar{K}_{n,i})-f(t_n,u(t_n)+c_i k \dot{u}(t_n))]+f(t_n+c_i k,u(t_n)+c_i k \dot{u}(t_n)) \nonumber \\
&&+s  A f(t_n,u(t_n))+s \varphi_1(s A_0)[Af(t_n+c_i k,u(t_n)+c_i k \dot{u}(t_n))-A f(t_n,u(t_n))]\nonumber \\
&&+s^2 \varphi_2(s A_0) A^2 f(t_n,u(t_n)).  \nonumber
\end{eqnarray}
From here,
\begin{eqnarray}
\bar{K}_{n,i}&=& u(t_n)+ c_i k  A u(t_n)+c_i^2 k^2 \varphi_2(k A_0) A^2 u(t_n) \nonumber \\
&&+k\sum_{j=1}^{i-1} a_{ij} \bigg[e^{(c_i-c_j)kA_0} [f(t_n+c_j k, \bar{K}_{n,j})-f(t_n,u(t_n))]+f(t_n,u(t_n)) \nonumber\\
&&\hspace{2cm} +(c_i-c_j)k \varphi_1((c_i-c_j)k A_0) A f(t_n,u(t_n))\bigg] \nonumber \\
&=&u(t_n)+c_i k\dot{u}(t_n)+O(k^2), \qquad i=1,\dots,s, \nonumber
\\
\bar{u}_{n+1}&=&u(t_n)+k A u(t_n)+ \frac{k^2}{2} A^2 u(t_n)+ k^3 \varphi_3(k A_0) A^3 u(t_n) \nonumber \\
&&+k \sum_{i=1}^s b_i \bigg[f(t_n+c_i k, u(t_n)+c_i k \dot{u}(t_n))+(1-c_i)k A f(t_n,u(t_n)) \nonumber \\
&&\hspace{2cm}+e^{(1-c_i)k A_0}[f(t_n+c_i k,\bar{K}_{n,i})-f(t_n+c_i k,u(t_n)+c_i k \dot{u}(t_n))]\nonumber \\
&& \hspace{2cm} +(1-c_i)k \varphi_1((1-c_i)k A_0) [A f(t_n+c_i k,u(t_n)+c_i k \dot{u}(t_n))-A f(t_n,u(t_n))] \nonumber \\
&& \hspace{2cm} +(1-c_i)^2 k^2 \varphi_2((1-c_i)k A_0) A^2 f(t_n,u(t_n))
\bigg] \nonumber\\
&=& u(t_n)+k[Au(t_n)+f(t_n,u(t_n))]\nonumber \\
&&+\frac{k^2}{2} [A^2 u(t_n)+f_t(t_n,u(t_n))+f_u(t_n,u(t_n)) \dot{u}(t_n)+ A f(t_n,u(t_n)) ]+O(k^3), \nonumber \\
&=& u(t_n)+k \dot{u}(t_n)+\frac{k^2}{2}  \ddot{u} (t_n) +O(k^3) \nonumber,
\end{eqnarray}
where (\ref{regloc3}) has been used, and so the first part of the theorem follows.

For the second part of the theorem, in a similar way to the proof of Theorem \ref{teorsalocal2}, but looking now at the term in $k^3$ and using the new stronger hypotheses, it can be proved that
\begin{eqnarray}
A_0^{-1} \rho_{n+1}&=&\frac{k^3}{6}A_0^{-1}\bigg[A^3 u+ f_{tt}+2 f_{tu} \dot{u}+ f_{uu} \dot{u}^2 + f_u A^2 u + f_u[f_t+f_u \dot{u} +A f] \nonumber \\
&&\hspace{1.5cm} +A f_t+A (f_u \dot{u})+A^2 f-\stackrel{\dots}{u}\bigg]+O(k^4),
\end{eqnarray}
and the term in bracket vanishes by differentiating (\ref{laibvp}) twice.
\end{proof}

\subsubsection{Local error of the full discretization}
 Following a similar proof to that of Theorem \ref{teorsalocalfd2}, the following result turns up:
\begin{theorem} Under the same hypotheses of the first part of Theorem \ref{teorsalocal3}, and assuming also that, for $t\in [0,T]$ and $\tau\in [0,\tau_0]$,
\begin{eqnarray}
A^l u(t)\in Z, \,\, l=0,1,2,3, \quad A^l f(t,u(t))\in Z, \,\,  l=1,2, \quad A f(t+\tau, u(t)+\tau \dot{u}(t))\in Z,
\nonumber
\end{eqnarray}
it happens that $\rho_{n,h}=O(k^3+k \varepsilon_h)$ where the constant in Landau notation is independent of $k$ and $h$. Moreover, under the additional hypotheses of the second part of Theorem \ref{teorsalocal3},
together with condition (\ref{cspch}), it follows that
$A_{h,0}^{-1} \rho_{n,h}=O(k^4+k \eta_h+k^2 \epsilon_h)$.
\label{teorsalocalfd3}
\end{theorem}

\subsubsection{Global error of the full discretization}
\label{secfd3}

This subsection is different from 4.1.3 in the fact that, for both Dirichlet and Robin/Neumann boundary conditions, numerical differentiation must be used to approximate the corresponding boundary values in (\ref{fd3}). Because of that, in order to assure convergence with this technique, we will have to ask that $k$ is sufficiently small with respect to $h^\gamma$, where $\gamma$ is the parameter which turns up when applying numerical differentiation, as stated in Remark \ref{rembc3}. Again, the classical argument would lead to a worse bound for the global error than in parabolic problems, when a summation-by-parts arguments can be used.
\begin{theorem}
Under the hypotheses of the first part of Theorem \ref{teorsalocalfd3}, if there exists a constant $C$ such that
\begin{eqnarray}
\frac{k}{h^\gamma} \le C,
\label{condcfl}
\end{eqnarray}
when considering Dirichlet boundary conditions, $e_{n,h}=O(k^2+\varepsilon_h+k \nu_h)$ and, with Robin/Neumann boundary conditions, $e_{n,h}=O(k^2+\varepsilon_h+k \nu_h+k \mu_{k,1})$, where $\nu_h$ and $\mu_{k,1}$ are the errors coming respectively from numerical differentiation in space and time according to Remark \ref{rembc3}. On the other hand, under the hypotheses of the second part of Theorem \ref{teorsalocalfd3},
but assuming also (\ref{spp}) and that
\begin{eqnarray}
&&\hspace{-0.5cm}f\in C^4([0,T]\times X, X), \,\,   u\in C^1([0,T], D(A^3))\cap C^5([0,T],X), \nonumber \\
&&\hspace{-0.5cm} A^l f(t,u(t))\in C^1([0,T],X),  \,l=1,2,\,\, A f(t+\tau,u(t)+\tau \dot{u}(t))\in C^1([0,T]\times [0,\tau_0] , X),  \nonumber \\
&&\hspace{-0.5cm} A^l \dot{u}(t)\in Z, \, l=0,1,2,3, \quad \frac{d}{dt} A^l f(t,u(t))\in Z, \, l=0,1,2, \quad t\in [0,T],\nonumber \\
&&\hspace{-0.5cm} \frac{d}{dt} A^l f(t+\tau,u(t)+\tau \dot{u}(t))\in Z, \, l=0,1, \quad \tau \in (0,\tau_0],
\label{regglob3}
\end{eqnarray}
it follows that, when considering Dirichlet boundary conditions, $e_{n,h}=O(k^3+k\varepsilon_h+\eta_h+k \nu_h)$ and, with Robin/Neumann boundary conditions, $e_{n,h}=O(k^3+k \varepsilon_h+\eta_h+k \mu_{k,1}+k \nu_h)$.
\label{teorsa3globalfd3}
\end{theorem}
\begin{proof}
For the proof, as in Theorem \ref{teorsaglobfd2}, we must consider the decomposition (\ref{decompg}) where $\bar{U}_{n+1,h}$ is calculated as $U_{n+1,h}$ but starting from $P_h u(t_n)$ and calculating the boundaries in (\ref{fd3}) in an exact way. In contrast, according to Table \ref{tabla}, when considering $U_{n+1,h}$, the boundaries in (\ref{fd3}) can just be calculated approximately.

More precisely, with Dirichlet boundary conditions, the terms on the boundary for the stages in (\ref{fd3}) can be calculated exactly. However, when calculating $U_{n+1,h}$, $\partial A^2 u(t_n)$ and $\partial A f(t_n,u(t_n))$ can just be calculated except for $O(\nu_h+\frac{e_{n,h}}{h^\gamma})$.
 Because of this,
\begin{eqnarray}
\lefteqn{U_{n+1,h}-\bar{U}_{n+1,h}= e^{k A_{h,0}} e_{n,h}+k^3 \varphi_3(k A_{h,0})C_h O(\nu_h+\frac{e_{n,h}}{h^\gamma})} \nonumber \\
&&+k \sum_{i=1}^s b_i \bigg[ e^{(1-c_i) k A_{h,0}}[f(t_n+c_i k, K_{n,h,i})-f(t_n+c_i k, \bar{K}_{n,h,i})]\nonumber \\
&& \hspace{2cm} +(1-c_i)^2 k^2 \varphi_2((1-c_i)k A_{h,0})C_h O(\nu_h+\frac{e_{n,h}}{h^\gamma})\bigg], \nonumber
\end{eqnarray}
where $K_{n,h,i}-\bar{K}_{n,h,i}=O(e_{n,h})$ as in the proof of Theorem \ref{teorsaglobfd2}. Therefore, using (\ref{recurf}) and (H2c),
\begin{eqnarray}
\lefteqn{\hspace{-1.5cm}U_{n+1,h}-\bar{U}_{n+1,h}= e^{k A_{h,0}} e_{n,h}+k^2 (\varphi_2(k A_{h,0})-\frac{1}{2}I)O(\nu_h+\frac{e_{n,h}}{h^\gamma})+O(k e_{n,h})} \nonumber \\
&&+k^2 \sum_{i=1}^s b_i(1-c_i) (\varphi_1(k A_{h,0})-I)O(\nu_h+\frac{e_{n,h}}{h^\gamma}), \nonumber
\end{eqnarray}
from what, using condition (\ref{condcfl}),
\begin{eqnarray}
e_{n+1,h}=e^{k A_{h,0}}e_{n,h}+O(k e_{n,h})+O(k^2 \nu_h)+\rho_{n+1,h}. \nonumber
\end{eqnarray}
The classical argument of convergence and the first part of Theorem \ref{teorsalocalfd3} leads then to the first result of this theorem for Dirichlet boundary conditions. For the second part, the second part of Theorem \ref{teorsalocalfd3} must be used, apart from (\ref{spp}) and the additional regularity (\ref{regglob3}).

On the other hand, with Robin/Neumann boundary conditions, there is some error when approximating the boundaries for both the stages and the numerical solution. More precisely, using Table \ref{tabla} and (\ref{fd3}),
\begin{eqnarray}
K_{n,h,i}-\bar{K}_{n,h,i}&=&e^{c_i k A_{h,0}} e_{n,h}+c_i^2 k^2 \varphi_2(c_i k A_{h,0}) C_h O(e_{n,h}) \nonumber \\
&&+k \sum_{j=1}^{i-1} a_{ij}[O(e_{n,h})+(c_i-c_j)k \varphi_1((c_i-c_j)k A_{h,0}) C_h O(e_{n,h})]
\nonumber \\
&=&e^{c_i k A_{h,0}} e_{n,h}+c_i k [\varphi_1(c_i k A_{h,0})-I] O(e_{n,h})\nonumber \\
&&+k \sum_{j=1}^{i-1} a_{ij}\bigg[O(e_{n,h})+ [e^{(c_i-c_j)k A_{h,0}}-I] O(e_{n,h})\bigg] \nonumber \\
&=&e^{c_i k A_{h,0}} e_{n,h}+O(k e_{n,h})=O(e_{n,h}), \nonumber
\end{eqnarray}
and then
\begin{eqnarray}
U_{n+1,h}-\bar{U}_{n+1,h}&=&e^{k A_{h,0}}e_{n,h}+k^3 \varphi_3(k A_{h,0})C_h O(\mu_{k,1}+\frac{e_{n,h}}{k}+\nu_h+\frac{e_{n,h}}{h^\gamma}) \nonumber \\
&&+k \sum_{i=1}^s b_i \bigg[ e^{(1-c_i)k A_{h,0}}[f(t_n+c_i k,K_{n,h,i})-f(t_n+c_i k,\bar{K}_{n,h,i})] \nonumber \\
&&\hspace{2cm}+(1-c_i)k \varphi_1((1-c_i)k A_{h,0})C_h O(k \mu_{k,1}+e_{n,h}) \nonumber \\
&&\hspace{2cm}+(1-c_i)^2 k^2 \varphi_2((1-c_i)k A_{h,0})C_h O(\mu_{k,1}+\frac{e_{n,h}}{k}+\nu_h+\frac{e_{n,h}}{h^\gamma})\bigg]
\nonumber \\
&=&e^{k A_{h,0}}e_{n,h}+k^2 [\varphi_2(k A_{h,0})-\frac{1}{2}I] O(\mu_{k,1}+\frac{e_{n,h}}{k}+\nu_h+\frac{e_{n,h}}{h^\gamma}) \nonumber \\
&&+k \sum_{i=1}^s b_i \bigg[ O(e_{n,h})+[e^{(1-c_i)k A_{h,0}}-I]O(k \mu_{k,1}+e_{n,h}) \nonumber \\
&&\hspace{2.5cm}+(1-c_i)k [\varphi_1(k A_{h,0})-I]O(\mu_{k,1}+\frac{e_{n,h}}{k}+\nu_h+\frac{e_{n,h}}{h^\gamma})
\bigg]. \nonumber
\end{eqnarray}
From this, under condition (\ref{condcfl}),
$$
e_{n+1,h}=e^{k A_{h,0}}e_{n,h}+O(k e_{n,h}+k^2 \mu_{k,1}+k^2 \nu_h)+\rho_{n+1,h},$$
so that, using the first part of Theorem \ref{teorsalocalfd3} and the classical argument of convergence, $e_{n,h}=O(k^2+\varepsilon_h+k \nu_h+k \mu_{k,1})$. Again, under the second set of hypotheses in Theorem \ref{teorsalocalfd3} and using (\ref{spp}) and the regularity (\ref{regglob3}), the finer result $e_{n,h}=O(k^3+k \varepsilon_h+\eta_h+k \mu_{k,1}+k \nu_h)$ can be achieved.
\end{proof}

\subsection{Searching for local order 4}
The idea is again to calculate the stages as in (\ref{etapast}) but with $v_n$, $w_{n,j}$ as in (\ref{timelocal3}), and $u_{n+1}$ through (\ref{unt}) but with $\tilde{v}_n$, $\tilde{w}_{n,j}$ satisfying
\begin{eqnarray}
&&\left\{ \begin{array}{rcl} \dot{\tilde{v}}_n(s)&=& A \tilde{v}_n(s), \\ \tilde{v}_n(0)&=&u_n, \\ \partial \tilde{v}_n(s)&=&\partial [u(t_n)+s A u(t_n)+\frac{s^2}{2} A^2 u(t_n)+\frac{s^3}{6} A^3 u(t_n)], \end{array}\right.
 \nonumber \\
 &&\left\{ \begin{array}{rcl} \dot{\tilde{w}}_{n,j}(s)&=& A \tilde{w}_{n,j}(s), \\ \tilde{w}_{n,j}(0)&=&f(t_n+c_jk,K_{n,j}), \\ \partial \tilde{w}_{n,j}(s)&=&\partial
 \bigg[f\bigg(t_n+c_j k,u(t_n)+c_j k A u(t_n)+\frac{c_j^2 k^2}{2} A^2 u(t_n) \nonumber \\
 &&\hspace{1cm} +k \sum_{r=1}^{j-1} a_{j,r}[f(t_n+c_rk, u(t_n)+c_r k \dot{u}(t_n))+(c_j-c_r)k A f(t_n,u(t_n))]\bigg) \nonumber
 \\
 &&\hspace{0.75cm}+s A f(t_n+c_j k, u(t_n)+c_j k \dot{u}(t_n))+\frac{s^2}{2} A^2 f(t_n,u(t_n))\bigg]. \end{array}\right. \label{timelocal4}
\end{eqnarray}
Again, after discretizing these problems in space and using the variation-of-constants formula, the following full discretation formulas arise:
\begin{eqnarray}
\lefteqn{K_{n,h,i}=e^{c_i k A_{h,0}}U_{n,h}+\sum_{l=1}^3 c_i^l  k^l  \varphi_l(c_i k A_{h,0}) C_h \partial A^{l-1} u(t_n)} \nonumber \\
&&+
k\sum_{j=1}^{i-1} a_{ij}\bigg[ e^{(c_i-c_j) k A_{h,0}}f(t_n+c_j k, K_{n,h,j}) \nonumber \\
&&\hspace{2cm}+(c_i-c_j) k \varphi_1((c_i-c_j)k A_{h,0})C_h \partial f(t_n+c_j k,u(t_n)+c_j k \dot{u}(t_n)) \nonumber \\
&&\hspace{2cm}+(c_i-c_j)^2 k^2 \varphi_2((c_i-c_j)k A_{h,0})C_h \partial A f(t_n,u(t_n))
\bigg], \nonumber \\
\lefteqn{U_{n+1,h}=e^{k A_{h,0}}U_{n,h}+\sum_{l=1}^4 k^l \varphi_l(k A_{h,0}) C_h \partial A^{l-1} u(t_n)} \nonumber \\
&&+k\sum_{i=1}^s b_i  \bigg[e^{(1-c_i)k A_{h,0}}f(t_n+c_i k, K_{n,h,i}) \nonumber \\
&& \hspace{2cm} +(1-c_i)k \varphi_1((1-c_i) kA_{h,0}) C_h \partial f\bigg(t_n+c_i k,u(t_n)+c_i k A u(t_n)+ \frac{c_i^2 k^2}{2} A^2 u(t_n)
\nonumber \\
&& \hspace{3.5cm}
+k \sum_{j=1}^{i-1} a_{ij}[f(t_n+c_j k, u(t_n)+c_j k \dot{u}(t_n))+(c_i-c_j)k A f(t_n,u(t_n))] \bigg)\nonumber \\
&& \hspace{2cm} +(1-c_i)^2 k^2 \varphi_2((1-c_i) kA_{h,0}) C_h \partial A f(t_n+c_i k,u(t_n)+c_i k \dot{u}(t_n)) \nonumber \\
&& \hspace{2cm} +(1-c_i)^3 k^3 \varphi_3((1-c_i) kA_{h,0}) C_h \partial A^2 f(t_n,u(t_n))
\bigg].\label{fd4}
\end{eqnarray}

\begin{remark}
\begin{table}
\begin{tabular}{|l|c|c|}
\hline
& Dirichlet & Robin/Neumann \\ \hline
$\partial u(t_n)$ & - & - \\ \hline
$\partial A u(t_n)/ \partial f(t_n,u(t_n))$& - & $O(e_{n,h})$ \\ \hline
$\partial f(t_n+c_i k, u(t_n)+c_i k \dot{u}(t_n))$ & - & $O(k \mu_{k,1}+e_{n,h})$ \\ \hline
$\partial A^2 u(t_n)/ \partial A f(t_n,u(t_n))$ & $O(\nu_h+\frac{e_{n,h}}{h^\gamma})$ & $O(\mu_{k,1}+\frac{e_{n,h}}{k}+\nu_h+\frac{e_{n,h}}{h^\gamma})$ \\ \hline
$\partial A f(t_n+c_i k, u(t_n)+c_i k \dot{u}(t_n))$ & $O(\nu_h+\frac{e_{n,h}}{h^\gamma}+\frac{k \mu_{k,1}}{h^\gamma})$ & $O(k \mu_{k,1}+k \mu_{k,2}+\frac{e_{n,h}}{k}+\nu_h+\frac{e_{n,h}}{h^\gamma})$ \\ \hline
$\partial A^3 u(t_n)/ \partial A^2 f(t_n,u(t_n))$ & $O(\nu_h+\frac{e_{n,h}}{k h^\gamma}+\frac{ \mu_{k,1}}{h^\gamma})$ & $O( \mu_{k,1}+ \mu_{k,2}+\frac{e_{n,h}}{k^2}+\nu_h+\frac{e_{n,h}}{k h^\gamma})$ \\ \hline
\end{tabular}
\caption{Errors which are committed at each step when approximating the corresponding boundary terms with the suggested technique to avoid order reduction, as justified in Remarks \ref{rembc}, \ref{rembc3} and \ref{rembc4}.}
\label{tabla}
\end{table}

Apart from the terms on the boundaries which were already necessary to achieve local order $3$, now
\begin{eqnarray}
\partial A^3 u(t_n), \quad \partial A^2 f(t_n,u(t_n)) \mbox{ and } \partial A f(t_n+c_i k,u(t_n)+c_i k \dot{u}(t_n))
\label{termsb4}
\end{eqnarray}
must also be calculated. Using (\ref{laibvp}) and simplifying notation,
\begin{eqnarray}
\partial A^3 u=\partial [\stackrel{\dots}{u}-(f_{tt}+2 f_{tu}\dot{u}+f_{uu}\dot{u}^2+f_u \ddot{u})]-\partial [A(f_t+f_u \dot{u})]-\partial A^2 f,
\label{a3b4}
\end{eqnarray}
where here everything is assumed to be evaluated either on $t_n$ or $(t_n,u(t_n))$. Now, in order to calculate $\partial [A(f_t+f_u \dot{u})]$ and $\partial A^2 f$, we can see that, with $A$ the second derivative in one dimension and $f$ like in (\ref{nonlinearterm}),
$$
A(f_t+f_u \dot{u})=\phi'''(u) u_x^2 \dot{u}+\phi''(u) u_{xx} \dot{u}+ 2 \phi''(u) u_x \dot{u}_x+\phi'(u) \dot{u}_{xx}+h_{txx},
$$
and it happens that $u_{xx}$ and $\dot{u}_{xx}$ can be calculated through
\begin{eqnarray}
u_{xx}=\dot{u}-\phi(u)-h, \quad \dot{u}_{xx}= \ddot{u}-\phi'(u) \dot{u}-h_t. \label{fint}
\end{eqnarray}
As for $A^2 f$,
$$A^2 f=\phi^{(4)}(u) u_x^4+ 6 \phi^{(3)}(u) u_x^2 u_{xx}+ 3 \phi''(u) u_{xx}^2+4 \phi''(u) u_x u_{xxx}+\phi'(u)u_{xxxx}+h_{xxxx},
$$
where $u_{xx}$ can be calculated as in (\ref{fint}) and
$$
u_{xxx}=\dot{u}_x-\phi'(u)u_x-h_x, \quad u_{xxxx}=\ddot{u}-\phi'(u) \dot{u}-h_t-\phi''(u) u_x^2-\phi'(u)(\dot{u}-\phi(u)-h)-h_{xx}.$$
Finally,
$$
A f(t_n+c_i k, u(t_n)+c_i k \dot{u}(t_n))=\phi''(u+c_i k \dot{u})(u_x+c_i k \dot{u}_x)^2+\phi'(u+c_i k \dot{u})(u_{xx}+c_i k \dot{u}_{xx})+h_{xx}(t_n+c_i k),
$$
where, in the right-hand-side $u$ and its derivatives are all evaluated at $t=t_n$ and, for $u_{xx}$, (\ref{fint}) can again be used.

Therefore, with Dirichlet boundary conditions, the boundary of every term is exactly calculable in terms of data except for $u_x$ and $\dot{u}_x$. Then, $u_x$ can be approximated as in Remark \ref{rembc3} and so $\dot{u}_x$ considering also space numerical differentiation over the exact values of $\dot{u}$ on the boundary and the approximated values of $\dot{u}$ in the interior of the domain, which can again be approximated by numerical differentiation in time from the values of the numerical solution. It can thus be deduced that, in this case, the approximation of $\dot{u}_x$ at the boundary differs from the exact in $O(\nu_h+\frac{e_{n,h}}{h k}+\frac{\mu_{k,1}}{h})$, where $\mu_{k,1}$ is the error which comes from the approximation of the first derivative if the values from which it is calculated were all exact. Because of this, the error of approximation of the first two terms in (\ref{termsb4}) behaves as $O(\nu_h+ \frac{e_{n,h}}{h k}+\frac{\mu_{k,1}}{h})$ and, for the last one, as $O(\nu_h+ \frac{e_{n,h}}{h k}+k\frac{\mu_{k,1}}{h})$ due to the factor $k$ multiplying $\dot{u}_x$. For a general operator $A$, we will assume all terms in (\ref{termsb4}) are calculated except for an error as that in Table \ref{tabla}
for some real value $\gamma$.

With Robin/Neumann boundary conditions, in one dimension and with $A$ the second derivative in space, $u|_{\partial \Omega},u_x|_{\partial \Omega}, \dot{u}|_{\partial \Omega}, \dot{u}_x|_{\partial \Omega},\ddot{u}|_{\partial \Omega}, \ddot{u}_x|_{\partial \Omega}$ will also be needed. (Notice that $\stackrel{\dots}{u}$ just appears linearly in (\ref{a3b4}) and therefore is given directly in terms of data through $\partial\stackrel{\dots}{u}=\stackrel{\dots}{g}$.) Then, $u|_{\partial \Omega}$ and $u_x|_{\partial \Omega}$ are calculated except for $O(e_{n,h})$ as in Remark \ref{rembc}; $\dot{u}|_{\partial \Omega}$ and $\dot{u}_x|_{\partial \Omega}$ except for $O(\mu_{k,1}+\frac{e_{n,h}}{k})$ as in Remark \ref{rembc3} and $\ddot{u}|_{\partial \Omega}$ and $\ddot{u}_x|_{\partial \Omega}$ in a similar way through numerical differentiation except for $O(\mu_{k,2}+\frac{e_{n,h}}{k^2})$, where $\mu_{k,2}$ comes from the error in the numerical approximation of the second derivative. For a general operator $A$, we thus assume that the error in the calculation of the first two  boundaries of (\ref{termsb4})  is
as written in the right-bottom part of Table \ref{tabla}
where the last two terms come from the possible error in the numerical approximation of some spatial derivatives of $u$ and $\dot{u}$.
For the boundary of the last term in (\ref{termsb4}), the error in the calculation, due to the factor $k$ multiplying $\dot{u}$, is also as specified in Table \ref{tabla}.
\label{rembc4}
\end{remark}
\begin{remark}
We notice that, if $\partial u(t)=\partial A u(t)= \partial A^2 u(t)=\partial A^3 u(t)=0$, from (\ref{laibvp}) it follows that $\partial f(t,u(t))=\partial A f(t,u(t))=\partial A^2 f(t,u(t))=0$. As in Remark \ref{remvan3}, $\partial f(t_n+c_i k, u(t_n+c_i k))=0$ differs from $\partial f(t_n+c_i k, u(t_n)+c_i k \dot{u}(t_n))$ in $O(k^2)$. Then, the stages in (\ref{fd4}) differ from those in the classical approach (\ref{lawsonca}) in
\begin{eqnarray}
k\sum_{j=1}^{i-1} a_{ij} (c_i-c_j)k \varphi_1((c_i-c_j)k A_{h,0})C_h O(k^2)=k \sum_{j=1}^{i-1} a_{ij} [e^{(c_i-c_j)k A_{h,0}}-I]A_{h,0}^{-1}C_h O(k^2)=O(k^3). \nonumber
\end{eqnarray}
On the other hand, $\partial f(t_n+c_i k, u(t_n+c_i k))=0$ also differs from the factor with the big parenthesis in (\ref{fd4}) in $O(k^3)$ and $\partial A f(t_n+c_i k, u(t_n+c_i k))=0$ differs from
$\partial A f(t_n+c_i k, u(t_n)+c_i k \dot{u}(t_n))$ in $O(k^2)$ . Because of all this, the difference in the numerical solution between (\ref{fd4}) and the classical approach is
\begin{eqnarray}
\lefteqn{k\sum_{i=1}^s b_i \bigg[ e^{(1-c_i)k A_{h,0}}O(k^3)+(1-c_i)k \varphi_1((1-c_i)k A_{h,0})C_h O(k^3)} \nonumber \\
&&\hspace{1cm}+(1-c_i)^2 k^2 \varphi_2((1-c_i)k A_{h,0})C_h O(k^2)\bigg] \nonumber \\
\lefteqn{=k\sum_{i=1}^s b_i \bigg[ e^{(1-c_i)k A_{h,0}}O(k^3)+[e^{(1-c_i)k A_{h,0}}-I]A_{h,0}^{-1}C_h O(k^3)}
\nonumber \\
&&\hspace{1cm}+(1-c_i) k [\varphi_1((1-c_i)k A_{h,0})-I]A_{h,0}^{-1}C_h O(k^2)\bigg]=O(k^4). \nonumber
\end{eqnarray}
This justifies, through Theorems \ref{teorsalocal4} and \ref{teorsalocalfd4}, that the local error with the classical approach, under these particular boundary conditions, behaves with order $4$ under the assumptions of those theorems.
\end{remark}

\subsubsection{Local error of the time semidiscretization}
With a similar proof to that of Theorems \ref{teorsalocal2} and \ref{teorsalocal3}, the following result follows:
\begin{theorem}
\label{teorsalocal4}
Under hypotheses (A1)-(A5) and (H1)-(H3), assuming also that, for every $t\in [0,T]$, $f(t,u(t))\in D(A^3)$; for small enough $\tau$ ($\tau\le \tau_0$), $f(t+\tau,u(t)+\tau \dot{u}(t)),f_t(t+\tau,u(t)+\tau \dot{u}(t)), f_u(t+\tau,u(t)+\tau \dot{u}(t))\dot{u}(t)\in D(A^2)$; for small enough $k$ and $\sigma$ ($k\le k_0, \sigma \le \sigma_0$) and $i=1,\dots,s$,
\begin{eqnarray}
&&f\bigg(t+c_i k, u(t)+c_i k A u(t)+\frac{c_i^2 \sigma^2}{2}A^2 u(t)\nonumber \\
&&\hspace{2cm}+\sigma \sum_{j=1}^{i-1} a_{ij}[f(t+c_j k, u(t)+c_j k \dot{u}(t))+(c_i-c_j)k A f(t,u(t))]\bigg)\in D(A), \nonumber \\
&&f_u\bigg(t+c_i k, u(t)+c_i k A u(t)+\frac{c_i^2 \sigma^2}{2}A^2 u(t)\nonumber \\
&&\hspace{2cm}
+\sigma \sum_{j=1}^{i-1} a_{ij}[f(t+c_j k, u(t)+c_j k \dot{u}(t))+(c_i-c_j)k A f(t,u(t))]\bigg)\in D(A); \nonumber
\end{eqnarray}
\begin{eqnarray}
&& u\in C([0,T], D(A^4))\cap C^4([0,T],X), \nonumber \\
&&f\in C^3([0,T]\times X,X), \nonumber \\
&& A^l f(t,u(t))\in C([0,T],X), \quad l=1,2,3, \nonumber \\
&& A^l f_t(t+\tau,u(t)+\tau \dot{u}(t)), A^l [f_u(t+\tau,u(t)+\tau \dot{u}(t))\dot{u}(t)] \in C([0,T]\times [0,\tau_0],X), \quad l=1,2, \nonumber
\end{eqnarray}
\begin{eqnarray}
&& A \bigg[ f_u\bigg(t+c_i k, u(t)+c_i k \dot{u}(t)+\frac{c_i^2 \sigma^2}{2}A^2 u(t)\nonumber \\
&&\hspace{2cm}+\sigma \sum_{j=1}^{i-1} a_{ij}[f(t+c_j k, u(t)+c_j k \dot{u}(t))+(c_i-c_j)k A f(t,u(t))]\bigg)
\nonumber \\
&&\hspace{1cm}
\cdot[c_i^2 \sigma A^2 u(t)+\sum_{j=1}^{i-1} a_{ij}[f(t+c_j k, u(t)+c_j k \dot{u}(t))+(c_i-c_j)k A f(t,u(t))]\bigg] \nonumber \\
&& \hspace{3cm}\in C([0,T]\times [0,k_0]\times [0,\sigma_0],X).
\label{regloc4}
\end{eqnarray}
and if the Runge-Kutta tableau corresponds to a method of classical order $\ge 3$, it follows that $\rho_n=O(k^4)$. Moreover, if $f\in C^4([0,T]\times X,X), u\in C^5([0,T],X)$, (\ref{cspc}) holds
and the Runge-Kutta tableau corresponds to a method of classical order $\ge 4$, it follows that $A_{0}^{-1} \rho_n=O(k^5)$.
\end{theorem}
\subsubsection{Local error of the full discretization}
In a similar way to the proof of Theorem \ref{teorsalocalfd3}, it follows that
\begin{theorem}
Under the same hypotheses of the first part of Theorem \ref{teorsalocal4}, and assuming also that, for $t\in [0,T]$, $i=1,\dots,s$ and $k\in [0,k_0]$,
\begin{eqnarray}
&&A^l u(t)\in Z, \,\, l=0,1,\dots,4, \quad A^l f(t,u(t))\in Z, \,\,  l=1,2,3, \nonumber \\
&& A^l f(t+c_i k, u(t)+c_i k \dot{u}(t))\in Z, \,\, l=1,2, \nonumber \\
&&A f\bigg(t+c_i k, u(t)+c_i k Au(t)+\frac{c_i^2 k^2}{2} A^2 u(t) \nonumber \\
&&\hspace{1cm}+k \sum_{j=1}^{i-1} a_{ij}[f(t+c_j k, u(t)+c_j k \dot{u}(t))+(c_i-c_j)k A f(t,u(t))]\bigg)\in Z,
\label{reg4}
\end{eqnarray}
it happens that $\rho_{n,h}=O(k^4+k \varepsilon_h)$ where the constant in Landau notation is independent of $k$ and $h$. Moreover, under the additional hypotheses of the second part of Theorem \ref{teorsalocal4}, together with condition (\ref{cspch}),
$A_{h,0}^{-1} \rho_{n,h}=O(k^5+k \eta_h)$.
\label{teorsalocalfd4}
\end{theorem}

\subsubsection{Global error of the full discretization}

In a similar way to Subsection \ref{secfd3},
\begin{theorem}
Under hypotheses of the first part of Theorem \ref{teorsalocalfd4} and assuming that (\ref{condcfl}) holds,
when considering Dirichlet boundary conditions, $e_{n,h}=O(k^3+\varepsilon_h+k \nu_h+k \mu_{k,1})$ and, with Robin/Neumann boundary conditions, $e_{n,h}=O(k^3+\varepsilon_h+k \nu_h+k \mu_{k,1}+k^2 \mu_{k,2})$, where $\nu_h$ and $\mu_{k,1}, \mu_{k,2}$ are the errors coming from numerical differentiation in space and time according to Remark \ref{rembc4}. On the other hand, under the hypotheses of the second part of Theorem \ref{teorsalocalfd4},
but assuming also (\ref{spp}) and that
\begin{eqnarray}
&&\hspace{-0.5cm}f\in C^5([0,T]\times X, X), \,\,   u\in C^1([0,T], D(A^4))\cap C^6([0,T],X), \nonumber \\
&&\hspace{-0.5cm} A^l f(\cdot,u(\cdot))\in C^1([0,T],X),  \,l=1,2,3, \nonumber \\
&&\hspace{-0.5cm}A^l f(t+c_i k,u(t)+c_i k \dot{u}(t))\in C^1([0,T] , X),  \,l=1,2, \, i=1,\dots,s, \, k \in (0,k_0], \nonumber \\
&&\hspace{-0.5cm}A f\bigg(\cdot+c_i k, u(\cdot)+c_i k Au(\cdot)+\frac{c_i^2 k^2}{2} A^2 u(\cdot) \nonumber \\
&&+k \sum_{j=1}^{i-1} a_{ij}[f(\cdot+c_j k, u(\cdot)+c_j k \dot{u}(\cdot))+(c_i-c_j)k A f(\cdot,u(\cdot))]\bigg)\in C^1([0,T],X), \nonumber \\
&&\hspace{-0.5cm} A^l \dot{u}(t)\in Z, \, l=0,1,\dots,4, \quad \frac{d}{dt} A^l f(t,u(t))\in Z, \, l=0,1,2,3, \quad t \in [0,T], \nonumber \\
&&\hspace{-0.5cm} \frac{d}{dt} A^l f(t+c_i k,u(t)+c_i k \dot{u}(t))\in Z, \, l=0,1,2, \quad k\in [0,k_0],\nonumber \\
&&\hspace{-0.5cm} \frac{d}{dt} A^l f\bigg(t+c_i k, u(t)+c_i k Au(t)+\frac{c_i^2 k^2}{2} A^2 u(t) \nonumber \\
&&+k \sum_{j=1}^{i-1} a_{ij}[f(t+c_j k, u(t)+c_j k \dot{u}(t))+(c_i-c_j)k A f(t,u(t))]\bigg)\in Z, \, l=0,1,
\label{regglob4}
\end{eqnarray}
it follows that, with Dirichlet boundary conditions, $e_{n,h}=O(k^4+k\varepsilon_h+\eta_h+k \nu_h+k \mu_{k,1})$ and, with R/N boundary conditions, $e_{n,h}=O(k^4+k \varepsilon_h+\eta_h+k \mu_{k,1}+k^2 \mu_{k,2}+k \nu_h)$.
\label{teorsa3globalfd4}
\end{theorem}
\begin{proof}
 As in  the proof of Theorem \ref{teorsa3globalfd3}, we must consider the decomposition (\ref{decompg}) and then study the difference $U_{h,n+1}-\bar{U}_{n+1}$ taking into account that the boundaries for $U_{h,n+1}$ in (\ref{fd4}) are just calculated approximately with an error which is given through Table \ref{tabla}.

More precisely, with Dirichlet boundary conditions,
\begin{eqnarray}
\lefteqn{U_{n+1,h}-\bar{U}_{n+1,h}= e^{k A_{h,0}} e_{n,h}+k^3 \varphi_3(k A_{h,0})C_h O(\nu_h+\frac{e_{n,h}}{h^\gamma})}
\nonumber \\
&& \hspace{1cm}+k^4 \varphi_4(k A_{h,0})C_h O(\nu_h+\frac{e_{n,h}}{h^\gamma k}+\frac{\mu_{k,1}}{h^\gamma}) \nonumber \\
&&+k \sum_{i=1}^s b_i \bigg[ e^{(1-c_i) k A_{h,0}}[f(t_n+c_i k, K_{n,h,i})-f(t_n+c_i k, \bar{K}_{n,h,i})]\nonumber \\
&& \hspace{2cm} +(1-c_i) k \varphi_1((1-c_i)k A_{h,0})C_h O(k^2 \nu_h+k^2 \frac{e_{n,h}}{h^\gamma})\nonumber \\
&& \hspace{2cm} +(1-c_i)^2 k^2 \varphi_2((1-c_i)k A_{h,0})C_h O(\nu_h+\frac{e_{n,h}}{h^\gamma}+ \frac{k\mu_{k,1}}{h^\gamma}) \nonumber \\
&& \hspace{2cm} +(1-c_i)^3 k^3 \varphi_3((1-c_i)k A_{h,0})C_h O(\nu_h+\frac{e_{n,h}}{h^\gamma k}+\frac{\mu_{k,1}}{h^\gamma})
\bigg], \label{form1}
\end{eqnarray}
where
\begin{eqnarray}
K_{n,h,i}-\bar{K}_{n,h,i}&=&e^{c_i k A_{h,0}}e_{n,h}+c_i^3 k^3 \varphi_3(k A_{h,0}) C_h O(\nu_h+\frac{e_{n,h}}{h^\gamma}) \nonumber \\
&&+k \sum_{j=1}^{i-1} a_{ij}\bigg[ e^{(c_i-c_j)k A_{h,0}}[f(t_n+c_j k,K_{n,h,j})-f(t_n+c_j k,\bar{K}_{n,h,j})]
\nonumber \\
&&\hspace{2cm}+(c_i-c_j)^2 k^2 \varphi_2((c_i-c_j)k A_{h,0}) C_h O(\nu_h+\frac{e_{n,h}}{h^\gamma})\bigg] \nonumber \\
&=&O(e_{n,h}+k^2 \nu_h), \nonumber
\end{eqnarray}
and, for the last equality, (\ref{recurf}), (H2c) and (\ref{condcfl}) have been used. Inserting this in (\ref{form1}) and using again (\ref{recurf}), (H2c) and (\ref{condcfl}), it follows that
\begin{eqnarray}
U_{n+1,h}-\bar{U}_{n+1,h}= e^{k A_{h,0}} e_{n,h}+ O(k^2 \nu_h+k e_{n,h}+k^2 \mu_{k,1}). \nonumber
\end{eqnarray}
From here,
\begin{eqnarray}
e_{n+1,h}=e^{k A_{h,0}}e_{n,h}+O(k e_{n,h})+O(k^2 \nu_h+k^2 \mu_{k,1})+\rho_{n+1,h}, \nonumber
\end{eqnarray}
and using a discrete Gronwall Lemma and the first part of Theorem \ref{teorsalocalfd4}, the first part of the theorem follows for Dirichlet boundary conditions. For the second part, the second part of Theorem \ref{teorsalocalfd4} must be used, apart from (\ref{spp}) and the additional regularity (\ref{regglob4}).

 As for Robin/Neumann boundary conditions, with similar arguments,
\begin{eqnarray}
K_{n,h,i}-\bar{K}_{n,h,i}&=&e^{c_i k A_{h,0}} e_{n,h}+c_i^2 k^2 \varphi_2(k A_{h,0}) C_h O(e_{n,h})+c_i^3 k^3 \varphi_3(k A_{h,0}) C_h O(\mu_{k,1}+\frac{e_{n,h}}{k}+\nu_h+\frac{e_{n,h}}{h^\gamma}) \nonumber \\
&&+k \sum_{j=1}^{i-1} a_{ij}\bigg[e^{(c_i-c_j) k A_{h,0}} [f(t_n+c_j k, K_{n,h,j})-f(t_n+c_j k, \bar{K}_{n,h,j})] \nonumber \\
&&\hspace{2cm}+(c_i-c_j) k \varphi_1((c_i-c_j)k A_{h,0}) C_h O(e_{n,h}+k(\mu_{k,1}+\frac{e_{n,h}}{k}))\nonumber \\
&&\hspace{2cm}+(c_i-c_j)^2 k^2  \varphi_2((c_i-c_j)k A_{h,0}) C_h O(\mu_{k,1}+\frac{e_{n,h}}{k}+\nu_h+\frac{e_{n,h}}{h^\gamma})\bigg] \nonumber \\
&=&e^{c_i k A_{h,0}} e_{n,h}+O(k e_{n,h}+k^2 \mu_{k,1}+k^2 \nu_h)=O(e_{n,h}+k^2 \mu_{k,1}+k^2 \nu_h), \nonumber
\end{eqnarray}
from what
\begin{eqnarray}
\lefteqn{U_{n+1,h}-\bar{U}_{n+1,h}} \nonumber \\
&=&e^{k A_{h,0}}e_{n,h}+
k^2 \varphi_2(k A_{h,0})C_h O(e_{n,h})+
k^3 \varphi_3(k A_{h,0})C_h O(\mu_{k,1}+\frac{e_{n,h}}{k}+\nu_h+\frac{e_{n,h}}{h^\gamma}) \nonumber \\
&&+k^4 \varphi_4(k A_{h,0})C_h O(\mu_{k,1}+\mu_{k,2}+\frac{e_{n,h}}{k^2}+\nu_h+\frac{e_{n,h}}{k h^\gamma}) \nonumber \\
&&+k \sum_{i=1}^s b_i \bigg[ e^{(1-c_i)k A_{h,0}}[f(t_n+c_i k,K_{n,h,i})-f(t_n+c_i k,\bar{K}_{n,h,i})] \nonumber \\
&&\hspace{2cm}+(1-c_i)k \varphi_1((1-c_i)k A_{h,0})C_h O(k e_{n,h}+k^2 \mu_{k,1}+k^2 \mu_{k,2}+k^2 \nu_h+\frac{k^2}{h^\gamma}e_{n,h}) \nonumber \\
&&\hspace{2cm}+(1-c_i)^2 k^2 \varphi_2((1-c_i)k A_{h,0})C_h O(k \mu_{k,1}+k \mu_{k,2}+\frac{e_{n,h}}{k}+\nu_h+\frac{e_{n,h}}{h^\gamma})\nonumber \\
&&\hspace{2cm}+(1-c_i)^3 k^3 \varphi_3((1-c_i)k A_{h,0})C_h O(\mu_{k,1}+ \mu_{k,2}+\frac{e_{n,h}}{k^2}+\nu_h+\frac{e_{n,h}}{k h^\gamma})
\bigg]
\nonumber \\
&=&e^{k A_{h,0}}e_{n,h}+O(k e_{n,h}+k^2 \mu_{k,1}+k^3 \mu_{k,2}+k^2 \nu_h). \nonumber
\end{eqnarray}
From this,
$$
e_{n+1,h}=e^{k A_{h,0}}e_{n,h}+O(k e_{n,h}+k^2 \mu_{k,1}+k^3 \mu_{k,2}+k^2 \nu_h)+\rho_{n+1,h},$$
so that, using the first part of Theorem \ref{teorsalocalfd4} and the classical argument of convergence, $e_{n,h}=O(k^3+\varepsilon_h+k \nu_h+k \mu_{k,1}+k^2 \mu_{k,2})$. Again, under the second set of hypotheses in Theorem \ref{teorsalocalfd4} and using (\ref{spp}) and the regularity (\ref{regglob4}), the finer result $e_{n,h}=O(k^4+k \varepsilon_h+\eta_h+k \mu_{k,1}+k^2 \mu_{k,2}+k \nu_h)$ is achieved.
\end{proof}

\section{Numerical results}
In this section, we show some numerical experiments which corroborate the previous results. For that, we have considered the following set of problems
\begin{eqnarray}
u_t(t,x)&=&u_{xx}(t,x)+u^2(t,x)+h(t,x), \quad x\in[0,1], \quad t \in [0,1], \nonumber \\
u(0,x)&=&u_0(x), \label{problemnr}
\end{eqnarray}
where the boundary conditions are either Dirichlet
\begin{eqnarray}
u(t,0)=g_0(t), \quad u(t,1)=g_1(t), \label{cdir}
\end{eqnarray}
or mixed (Dirichlet/Neumann),
\begin{eqnarray}
u(t,0)=g_0(t), \quad u_x(t,1)=g_1(t), \label{cdn}
\end{eqnarray}
and where $h(t,x)$, $u_0(x)$ and $g_0(t),g_1(t)$ are such that the exact solution of the problem is either
\begin{eqnarray}
u(x,t)=x(x-1) \cos(x+t) \quad \mbox{or} \quad  u(x,t)=\cos(x+t).
\label{sols}
\end{eqnarray}
In the first place, as space discretization we have considered the symmetric 2nd-order difference scheme for which, in the Dirichlet case (\ref{cdir}),
\begin{eqnarray}
A_{h,0}=\mbox{tridiag}(1,-2,1)/h^2, \quad C_h [g_0(t), g_1(t)]^T=[g_0(t), 0, \dots, 0, g_1(t)]^T/h^2,
\label{Ah0dir}
\end{eqnarray}
and, in the Dirichlet/Neumann case (\ref{cdn}),
\begin{eqnarray}
A_{h,0}=\left[ \begin{array}{ccccc} -2 & 1 & 0 &\dots & 0 \\ 1 & -2 & 1 & & \\ & \ddots & \ddots & \ddots & \\ & &1 & -2 & 1 \\ & & 0 & 2 & -2 \end{array} \right], \quad C_h \left[\begin{array}{c} g_0(t) \\ g_1(t) \end{array} \right]= \left[\begin{array}{c} \frac{1}{h^2}g_0(t) \\0 \\ \vdots \\ 0 \\ \frac{2}{h} g_1(t)\end{array}\right].
\label{Ah0cdn}
\end{eqnarray}
All differential problems (\ref{problemnr}) with boundary conditions (\ref{cdir})-(\ref{cdn}) satisfy hypotheses (A1)-(A5) with $X=C([0,1])$ and the respective space discretizations (\ref{Ah0dir}) and (\ref{Ah0cdn}) satisfy hypotheses (H1)-(H3), as it was justified in \cite{acrnl} for $Z=C^4([0,1])$, $\varepsilon_h,\eta_h$ being $O(h^2)$ for (\ref{Ah0dir}) and  $\varepsilon_h=O(h)$, $\eta_h=O(h^2)$ for (\ref{Ah0cdn}). Besides, the considered solutions (\ref{sols}) and $f$ are so smooth that all conditions of regularity in the paper are satisfied. Finally, although we do not provide a proof for the conditions (\ref{acotses}) and (\ref{cspch}) with the maximum norm, it can be numerically verified that those conditions hold uniformly on $h$ for $A_{h,0}$ in (\ref{Ah0dir}) and (\ref{Ah0cdn}). Moreover, (\ref{cspch}) points out that (\ref{cspc}) is also satisfied in the continuous case for the regular functions $u(t)$ which are considered, although its proof is not an aim of this paper either.

\subsection{Second-order method}
\begin{table}[h]
\begin{center}
\begin{tabular}{|c|c|c|c|c|} \hline
k & 1e-3 & 5e-4 & 2.5e-4 & 1.25e-4 \\ \hline
Local error & 9.7450e-4 & 4.8158e-4 & 2.3693e-4 & 1.1580e-4 \\ \hline
Order & & 1.02 & 1.02 & 1.03 \\ \hline
Global error & 1.3461e-3 & 6.6579e-4 & 3.2779e-4 & 1.6034e-4 \\ \hline
Order & & 1.02 & 1.02 & 1.03 \\ \hline
\end{tabular}
\label{t1}
\caption{Local and global error when integrating Dirichlet problem with vanishing boundary conditions with the classical approach (\ref{lawsonca}) associated to the second-order method (\ref{RK2}).}
\end{center}

\end{table}

\begin{table}[h]
\begin{center}
\begin{tabular}{|c|c|c|c|c|} \hline
h & k=1e-3 & k=5e-4 & k=2.5e-4 & k=1.25e-4 \\ \hline
2e-3 & 1.2404e+2& 6.1563e+1 &3.0339e+1 &1.4751e+1 \\ \hline
1e-3 & 4.9902e+2& 2.4903e+2& 1.2404e+2& 6.1563e+1  \\ \hline
5e-4 & 1.9990e+3& 9.9902e+2& 4.9902e+2& 2.4903e+2\\ \hline
\end{tabular}
\caption{Local error when integrating Dirichlet problem with non-vanishing boundary conditions with the classical approach (\ref{lawsonca}) associated to the second-order method (\ref{RK2}).}
\end{center}
\label{t2}
\end{table}

\begin{table}[h]
\begin{center}
\begin{tabular}{|c|c|c|c|c|} \hline
h & k=1e-3 & k=5e-4 & k=2.5e-4 & k=1.25e-4 \\ \hline
2e-3 & 6.7023e+1& 3.3264e+1& 1.6394e+1& 7.9729e+0\\ \hline
1e-3 & 2.6962e+2& 1.3455e+2& 6.7023e+1& 3.3264e+1  \\ \hline
5e-4 & 1.0801e+3& 5.3977e+2& 2.6962e+2& 1.3455e+2\\ \hline
\end{tabular}
\caption{Global error when integrating Dirichlet problem with non-vanishing boundary conditions with the classical approach (\ref{lawsonca}) associated to the second-order method (\ref{RK2}).}
\end{center}
\label{t3}
\end{table}

\begin{table}[h]
\begin{center}
\begin{tabular}{|c|c|c|c|c|} \hline
k & 1e-03 & 5e-4 & 2.5e-4 & 1.25e-4 \\ \hline
Local error & 1.5664e-7& 3.9176e-8& 9.7933e-9& 2.4473e-9\\ \hline
Order & &2.00 & 2.00& 2.00\\ \hline
Global error & 8.2929e-7& 2.0714e-7& 5.1712e-8& 1.2903e-8\\ \hline
Order & & 2.00& 2.00& 2.00\\ \hline
\end{tabular}
\caption{Local and global error when integrating Dirichlet problem with nonvanishing boundary conditions with the suggested approach (\ref{fd2}) associated to the second-order method (\ref{RK2}), with which no numerical differentiation is required.}
\end{center}
\label{t4}
\end{table}

\begin{table}[h]
\begin{center}
\begin{tabular}{|c|c|c|c|c|} \hline
k & 8e-3 & 4e-3 & 2e-3 & 1e-3 \\ \hline
Local error & 1.3126e-7 & 1.6797e-8 & 2.1350e-9 & 2.7857e-10 \\ \hline
Order & & 2.97 & 2.98 & 2.94 \\ \hline
Global error & 5.9892e-7 & 1.4972e-7 & 3.7367e-8 & 9.2309e-9 \\ \hline
Order & & 2.00 & 2.00 & 2.02 \\ \hline
\end{tabular}
\caption{Local and global error when integrating  Dirichlet problem with nonvanishing boundary conditions with the suggested approach (\ref{fd3}) associated to the second-order method (\ref{RK2}), for which numerical differentiation is required.}
\end{center}
\label{t5}
\end{table}

We first show the results which are obtained when integrating problem (\ref{problemnr}) associated to the Dirichlet boundary conditions with the Lawson method which is constructed with the second-order RK tableau
\begin{eqnarray}
\begin{array}{c|cc} 0 & 0 & \\ 1 & 1 & 0 \\ \hline & \frac{1}{2} & \frac{1}{2} \end{array}. \label{RK2}
\end{eqnarray}
When considering the solution in (\ref{sols}) which vanishes at the boundary, the classical approach (\ref{lawsonca}) shows local and global order $1$ in time, as shown in Table 2 for $h=5\times 10^{-4}$, for which the error in space is negligible. (This corroborates Theorems \ref{teorcalocal} and \ref{teorcaglobal}.) However, when the solution does not vanish at the boundary, although the local and global orders are still $1$, the errors are very big and even grow when $h$ diminishes. This was justified in Theorems \ref{teorcalocal3} and \ref{teorcaglobal3} and can be observed in Tables 3 and 4. However, that bad behaviour can be solved by using the suggested approach (\ref{fd2}), where every term on the boundary can be calculated in terms of data, without resorting to numerical differentiation. In such a way, local and global order 2 is obtained in Table 5  with $h=5\times 10^{-4}$, for which the error in space is again negligible and does not grow with $h$. This corroborates Theorems \ref{teorsalocalfd2} and \ref{teorsaglobfd2}. On the other hand, with this method, it is even possible to achieve local order $3$ with formula (\ref{fd3}), although numerical differentiation is required to approximate the boundary of the first derivative in space of the exact solution, as it is thoroughly explained in Remark \ref{rembc3}. For that, we have considered the $2$-BDF formula, for which $\nu_h=O(h^2)$ and, as already predicted by the first part of Theorems \ref{teorsalocalfd3} and \ref{teorsa3globalfd3}, Table 6 shows local order near $3$ and global order $2$, but with a size of errors quite smaller than those of Table 5.

\subsection{Third-order method}
\begin{table}[h]
\begin{center}
\begin{tabular}{|c|c|c|c|c|} \hline
k & 0.2 & 0.1 & 0.05 & 0.025 \\ \hline
Local error & 9.7639e-1 & 9.8964e-1 & 9.9108e-1 & 9.8877e-1 \\ \hline
Order & & -0.02 & -0.00 & 0.00 \\ \hline
Global error & 5.3822e-1 & 5.3736e-1 & 5.3613e-1 & 5.3439e-1\\ \hline
Order & & 0.00 & 0.00 & 0.00 \\ \hline
\end{tabular}
\caption{Local and global error when integrating mixed D/N problem with nonvanishing boundary conditions with the classical approach associated to the third-order method (\ref{RK3}), $h= 10^{-3}$.}
\end{center}
\label{t6}
\end{table}
\begin{table}[h]
\begin{center}
\begin{tabular}{|c|c|c|c|c|} \hline
k & 0.2 & 0.1 & 0.05 & 0.025 \\ \hline
Local error & 1.3911e-3 & 1.7489e-3& 2.1806e-5& 2.7212e-6\\ \hline
Order & &2.99 & 3.00& 3.00\\ \hline
Global error & 1.5136e-3& 2.3369e-4& 2.9913e-5& 3.6533e-6\\ \hline
Order & &2.70 &2.97 & 3.03\\ \hline
\end{tabular}
\caption{Local and global error when integrating mixed D/N problem with nonvanishing boundary conditions with the suggested approach (\ref{fd3}) associated to the third-order method (\ref{RK3}), for which numerical differentiation is required, $h= 10^{-3}$.}
\end{center}
\label{t7}
\end{table}

In this subsection we show the results which are obtained when considering (\ref{problemnr}) associated to the mixed Dirichlet/Neumann boundary conditions in (\ref{cdn}) and integrating it in time with the Lawson method associated to the third order Heun method
\begin{eqnarray}
\begin{array}{c|ccc} 0 &  & &\\ \frac{1}{3} & \frac{1}{3} & & \\  \frac{2}{3} & 0  & \frac{2}{3} &  \\  \hline & \frac{1}{4} & 0 & \frac{3}{4} \end{array}. \label{RK3}
\end{eqnarray}
We have centered on the solution of (\ref{sols}) which does not vanish at the boundary. The classical approach shows no convergence either on the local or global error where the timestepsize diminishes, as it is justified in Theorems \ref{teorcalocal2} and \ref{teorcaglobal2}, and it is shown in Table 7. (Notice the different behaviour with respect to the classical approach in Tables 3 and 4. Here the errors do not diminish with $k$ but, although not shown here for the sake of brevity, they neither grow when $h$ diminishes as it happens in those tables. This is due to the fact that now every $c_i$ is different from $1$). However, we can get local and thus global order $3$ with our modified approach (\ref{fd3}), by calculating the terms on the boundary following again Remark \ref{rembc3}. For the Dirichlet boundary condition, we have used numerical differentiation in space with the $2$-BDF formula and, for the Neumann one, numerical differentiation in time with the $3$-BDF scheme. In such a case, Theorem \ref{teorsalocalfd3} as well as the second part of Theorem \ref{teorsa3globalfd3} apply, with $\nu_h=O(h^2)$ and $\mu_{k,1}=O(k^3)$. Therefore, when the error in space is negligible, order $3$ in the timestepsize should be seen when, as Table 8 corroborates.

\subsection{Fourth-order method}
\begin{table}[h]
\begin{center}
\begin{tabular}{|c|c|c|c|c|} \hline
k & 0.2 & 0.1 & 0.05 & 0.025 \\ \hline
Local error & 1.8356e-4& 1.0396e-5& 6.1679e-7& 3.7509e-8\\ \hline
Order & &4.14 & 4.08& 4.04\\ \hline
Global error & 1.9072e-4& 9.3054e-6& 5.4646e-7& 3.5333e-8\\ \hline
Order & & 4.36 & 4.09& 3.95\\ \hline
\end{tabular}
\caption{Local and global error when integrating Dirichlet problem with nonvanishing boundary conditions with the suggested approach (\ref{fd4}) associated to the fourth-order method (\ref{RK4}), when the terms on the boundary are exactly provided, $h= 5\times
10^{-4}$.}
\end{center}
\label{t8}
\end{table}

\begin{table}[h]
\begin{center}
\begin{tabular}{|c|c|c|c|c|} \hline
k & 2.5e-2 & 1.25e-2 & 6.25e-3 & 3.125e-3 \\ \hline
Local error & 3.4537e-8& 2.0441e-9& 1.1954e-10& 6.8247e-12\\ \hline
Order & & 4.08& 4.10& 4.13\\ \hline
Global error & 3.3314e-8 & 2.0054e-9 & 1.1968e-10 & 7.0050e-12 \\ \hline
Order & & 4.05 & 4.07 & 4.09  \\ \hline
\end{tabular}
\caption{Local and global error when integrating Dirichlet problem with nonvanishing boundary conditions with the suggested approach (\ref{fd4}) associated to the fourth-order method (\ref{RK4}), when the terms on the boundary are  calculated through numerical differentiation and  Gauss-Lobatto collocation space discretization is used.}
\end{center}
\label{t9}
\end{table}

Finally, we show that local and global order $4$ can be obtained when integrating in time with the Lawson method associated to the fourth-order RK method
\begin{eqnarray}
\begin{array}{c|cccc} 0 &  & & &\\ \frac{1}{3} & \frac{1}{3} & & & \\  \frac{2}{3}   & -\frac{1}{3} & 1 & &  \\ 1 & 1 & -1 & 1 &\\  \hline & \frac{1}{8} & \frac{3}{8} & \frac{3}{8} & \frac{1}{8} \end{array}. \label{RK4}
\end{eqnarray}
Nevertheless, we point out that it is necessary to take condition (\ref{condcfl}) into account. In our problem (\ref{problemnr}), $\gamma=1$, as it was justified in Remarks \ref{rembc3} and \ref{rembc4}. For the sake of brevity, we will center on the Dirichlet boundary condition in (\ref{cdir}), and we will directly integrate that problem with the suggested formulas (\ref{fd4}) by inserting the needed boundaries in an exact way from the known solution. As Table 9 shows, local and global order $4$ are observed in that way. However, when not knowing the exact solution, those boundaries must be calculated in terms of data following Remark \ref{rembc4}. For that, we have considered again the $3$ (resp. 2)-BDF formula for the numerical differentiation in time (resp. in space) and, according to Theorems \ref{teorsalocalfd4} and \ref{teorsa3globalfd4}, global order $4$ in the timestepsize should be observed when the error in space is negligible and (\ref{condcfl}) holds for some constant $C$. However, as the error in space is just of second order, in order that the error in space is negligible with respect to that in time,  $h$ must be quite small with respect to $k$, and then the global error exploits to infinity with the parameters of Table 9 because condition (\ref{condcfl}) is not satisfied for a suitable constant $C$.

In spite of all this, the problem can be solved by considering a more accurate space discretization. Thus we have considered a Gauss-Lobatto collocation space discretization with $17$ nodes, for which the error in space is nearly of the order of rounding errors for this problem. Besides, the space grid is quite moderate, so condition (\ref{condcfl}) is very weak in this case. Considering then numerical differentiation in time as before and numerical differentiation in space through the derivation of the corresponding collocation polynomials, the results in Table 10 are obtained, where both local and global order $4$ are achieved.

\section*{Acknowledgements}
This research has been supported by Ministerio de Ciencia e
Innovaci\'on and Regional Development European Funds through project MTM2015-66837-P and by Junta de Castilla y Le\'on and Feder through project VA024P17.

\end{document}